 \documentclass[11pt,reqno]{amsart}
\usepackage{amssymb,mathrsfs,graphicx,enumerate,subfigure}
\usepackage{amsmath,amsfonts,amssymb,amscd,amsthm,bbm}
\numberwithin{figure}{section}
\newcommand{\lbl}[1]{\label{#1}}
\usepackage{tikz}
\tikzset{
	block/.style={
		draw, 
		rectangle, 
		minimum height=1.5cm, 
		minimum width=3cm, align=center
	}, 
	line/.style={->,>=latex'}
}
\usetikzlibrary{positioning,arrows}
\usepackage{caption}
\usetikzlibrary{arrows.meta}

\usepackage[retainorgcmds]{IEEEtrantools}
\usepackage{colortbl}
\usepackage{subfigure,here}
\newcommand{\bbr}{\mathbb R}
\title[ Stochastic Cucker-Smale  system  in  a harmonic field]{Flocking and concentration behavior for the stochastic Cucker-Smale  system  in  a harmonic field}

\author[Du]{Linglong Du}
\address[Linglong Du]{\newline Department of Applied Mathematics and Institute for Nonlinear Science\newline
	Donghua University, Shanghai, P.R. China}
\email{matdl@dhu.edu.cn}

\author[Zhou]{Xinyun Zhou}
\address[Xinyun Zhou]{\newline Department of Applied Mathematics  \newline  Donghua University, Shanghai, P.R. China} 
\email{15002192230@163.com}

\newtheorem{theorem}{Theorem}[section]
\newtheorem{lemma}{Lemma}[section]

\newtheorem{remark}{Remark}[section]

\newtheorem{definition}{Definition}[section]

\begin{document}
	\date{\today}
	\date{\today}

	\begin{abstract}
 We consider  the Cucker-Smale  system with multiplicative noise in  a harmonic  potential field and investigate the effect of harmonic  potential field.  In the presence of external potential force, the system is expected to emerge into  almost surely  velocity flocking  and  spatial concentration, due to the alignment mechanism, confining harmonic  potential field and multiplicative noise. By constructing a stochastic Lyapunov functional, we derive a sufficient condition for the almost surely flocking and concentration behavior  for the  stochastic particle model, and verify it numerically.  Then, we discuss the  flocking and concentration behavior  for the  mean field Vlasov-type kinetic model.  Moreover, a rigorous analysis of the uniform mean-field limit  estimate for the limit process from the stochastic model to the kinetic one is provided. 
	\end{abstract}
	
\maketitle \centerline{\date}

\section{Introduction} \label{sec:1}
\setcounter{equation}{0}
Collective behaviors of  particle-based  systems have been widely studied in recent years. 
 When systems have a finite number of particles, it is often argued that the microscopic approach is the appropriate one. The  particle Cucker-Smale (C-S) model is one of many microscopic attempts to represent
  such phenomena.  It was originally introduced by Cucker and Smale \cite{CS}, and studied by many authors, for example, \cite{HL,HE, ME, P} for the particle system with different communication kernels,  \cite{LW} for the  particle system with time delay, \cite{AH, H1, PL} for the particle system with different noises,   \cite{CF} for the kinetic C-S model. However, in many realistic scenarios, particles driven by alignment are also subject to environment forces. 
  Recently, \cite{RE} introduced the following C-S model with the  convex potential force $U$:
  \begin{align}
	\left\{\begin{aligned}\label{CS}
		&{\dot x}_t^i  =v_t^i ,   \quad  t>0,  \quad  1 \leq i \leq N,  \\
		&{\dot v}_t^i = \frac{\kappa}{N}\sum_{j=1}^N \psi(|x_t^j-x_t^i|) (v_t^j-v_t^i) -\nabla U(x_t^i).
	\end{aligned}\right.
\end{align}

  In this paper,  similar to the work studied in \cite{H1}, we add the stochastic influences  to  the deterministic dynamical system   \eqref{CS} with the quadratic potential $U(x)=\frac{1}{2}x^2$, 
  and consider the following stochastic system
     \begin{align}
	\left\{\begin{aligned}\label{SCS1}
		&dx_t^i =v_t^i dt,   \qquad  t>0,  \quad  1 \leq i \leq N,  \\
		&dv_t^i = \frac{\kappa}{N}\sum_{j=1}^N \psi(|x_t^j-x_t^i|) (v_t^j-v_t^i) dt-x_t^idt+\sqrt{2\sigma} (v_t^i-v_t^c) dW_t.
	\end{aligned}\right.
\end{align}
  $(x_t^i ,v_t^i)\in\mathbb{R}^{2d}$ denote the velocity and position of the $i$th agent. The noise $W_t$ is the same Brownian  motion in all directions of  $\mathbb{R}^d$. The diagonal diffusion matrix  $\sqrt{2\sigma} (v_t^i-v_t^c)$ tells us that ``noise intensity" depends  on the localization of the velocity in a simple way, here $\sigma$ is a positive constant. $\psi(|x_t^j-x_t^i|)$ is  a communication weight function and satisfies the symmetry condition and translation invariance. $|\cdot |$ denotes the standard  $\ell^2$-norm in $\mathbb{R}^{d}$.
  
 When the number of particles $N$ is excessively large, it becomes increasingly difficult to follow the dynamics of each individual agent. Hence, instead of simulating the
behavior of each individual agent, we would like to describe the kinetic collective behavior
encoded by the density distribution whose evolution is governed by one sole mesoscopic partial differential equation.
Applying the BBGKY hierarchy,  one can  derive the following  corresponding kinetic model from  \eqref{SCS1}:
\begin{align}\label{km}
	\begin{split}
		&\partial_t f +  v \cdot\nabla_{x}  f + \nabla_{v} \cdot (L[f]f)-x \cdot \nabla_{v}f=\sigma\Delta_{v}(|v-v_c|^2f) ,   \\
		&L[f](x,v,t) := -\kappa\int_{\mathbb{R}^{2d}} \psi \left(| x- y |\right) \left(v-v_\ast\right) f(y,v_\ast, t) dv_\ast dy, \quad  x,v \in \mathbb{R}^d,\ t>0,\\
	\end{split}
\end{align}
subject to suitable initial configuration 
\begin{align}\label{k1}
	\begin{split}
		&f(x,v,0)=f_0(x,v).\\
	\end{split}
\end{align}

In this paper,  we address the following  question: \newline

(Q):~ With the confining harmonic potential, can we drive a sufficient condition to rigorously verify the formation of the velocity flocking and  spatial concentration result both for the stochastic  particle model  \eqref{SCS1} and related mean-field  Vlasov-type kinetic model \eqref{km}?  \newline

 We exploit the above question from both the analytic and numerical views.  First we derive a sufficient framework  leading to the  formation of the  almost surely flocking in velocity and  spatial concentration due to the alignment  mechanism, confining potential and multiplicative noise.   Without the confining potential field, the authors proved the velocity flocking result by  deriving  a differential inequality system and solving a geometric Brownian motion equation for the velocity  directly \cite {AH}. The result in their paper also implied that the multiplicative noise  could give the strong flocking result, compared with the case of  the additive noise.   However,  system  \eqref{SCS1} is coupled by the existence of confining harmonic field. We use the  Lyapunov functional method and define a related  differential operator to get the almost surely velocity flocking  and spatial concentration result. By using Euler's method, we verify the analytic results numerically. 
 
 Second, we study the long-time behavior for the related mean-field  Vlasov-type kinetic model  \eqref{km}. We introduce  a Lyapunov functional $\tilde{\mathcal L}[f]$ (see \eqref{L}),  which is equivalent  to the standard Lyapunov functional $\mathcal L[f]$( see  \eqref{L1}) measuring the  position and velocity variances of the kinetic density function $f$. We exploit  a sufficient condition for the flocking and concentration result  for the mean-field  Vlasov-type kinetic model \eqref{km}. Due to the confining potential term, the standard Lyapunov functional  is not enough to  derive the Gr\"onwall's  type inequality for it. Therefore we consider the equivalent functional  containing the cross term
$\int_{\mathbb{R}^{2d}} (x-x_c)\cdot(v-v_c)fdxdv$.  We also address the existence of the classical solution for the kinetic model in a weighted Sobolev space.

Third, since we are dealing with the Vlasov-type kinetic equation corresponding to the  stochastic particle C-S model with a harmonic potential field, we  discuss the mean-field limit  and provide the rigorous analysis  for it.  In \cite{BC}, rigorous finite-in-time  mean-field limit has been derived  from the  C-S particle system with additive noise to the corresponding C-S Vlasov-type equation.  However, to get the uniform-in-time  mean-field limit estimate for our model, we need to perform a more rigorous analysis, considering the coupling difficulty brought by  the confining potential filed. We also utilize the detailed information of the McKean process, which was developed  in \cite{HKR}.  Therefore, by making use of  the flocking estimate for the limit process, we establish  the proof of  the uniform-in-time mean-field limit.

We introduce the following definition.
\begin{definition} \label{def}
The stochastic system has an asymptotic  strong stochastic flocking  in velocity and concentration in position if the position and velocity processes  $\{x_t^i, v_t^i\}$($i=1,\cdots,N$) satisfy the following  condition:
       For $1\leq i,j\leq N$, the differences of all pairwise position and velocity processes go to zero asymptotically,
        \begin{equation*} 
        \lim_{t \to \infty} (|x_t^i-x_t^j|+ |v_t^i-v_t^j|)=  0,\quad a.s.
        \end{equation*}
   \end{definition}
    The rest of this paper is organized as follows. In Section \ref{sec:2}, we study the  velocity flocking and  spatial concentration result for the stochastic particle model  analytically and  numerically. 
     In Section \ref{sec:3}, we present the well-posedness and long-time behavior of the kinetic C-S Vlasov-type equation.  In Section \ref{sec:4}, we prove the uniform-in-time mean-field  limit from the stochastic  particle C-S system to the kinetic C-S Vlasov-type equation. 
\section{Stochastic particle system}\label{sec:2}
\setcounter{equation}{0}%: radially symmetric communication rate function}
 In this section, we study the  velocity flocking and  spatial concentration result for the stochastic particle model  \eqref{SCS1} analytically and  numerically. We first introduce a macro-micro decomposition to decompose the system into two parts: one system  
describes the macroscopic dynamics and the other system describes the microscopic fluctuations. 

\subsection{A Macro-Micro decomposition}
 For the stochastic flocking estimate, we introduce macro (ensemble average) process 
 \begin{equation*}
x_t^c = \frac{1}{N}\sum_{i=1}^Nx_t^i,\quad v_t^c = \frac{1}{N}\sum_{i=1}^N v_t^i
\end{equation*}
%be the center of position and velocity, 
and  micro  (fluctuation)  process
\begin{equation*}
\hat{x}_t^i = x_t^i-x_t^c,\quad \hat{v}_t^i = v_t^i-v_t^c.
\end{equation*}
%be the deviations. 
Then $\sum\limits_{i=1}^N\hat{x}_t^i=\sum\limits_{i=1}^N \hat{v}_t^i= 0$. 

Averaging over $i$ in (\ref{SCS1}) gives the evolution of $(x_t^c,v_t^c)$:
\begin{equation}\label{macro1}
\left\{\begin{split}
& d x_t^c = v_t^cdt,\\
& dv_t^c = - x_t^c dt.
\end{split}\right.
\end{equation}
Given the  deterministic initial configuration $(x^c_0, v^c_0)$, one can get the dynamics of the macroscopic variables $(x_t^c,v_t^c)$, which satisfy the harmonic oscillator motion as follows:
 \begin{align*}
 x^c_t=x^c_0\cos t+v^c_0\sin t,  \quad
 v^c_t=v^c_0\cos t-x^c_0\sin t, \quad a.s.
       \end{align*} 

 Next, we subtract \eqref{macro1} from \eqref{SCS1} to derive the evolution of perturbation $(\hat{x}_t^i,\hat {v}_t^i)$:
\begin{equation}
\left\{\begin{split}\label{micro}
&d\hat{x}_t^i = \hat{v}_t^i dt,\\
&d\hat{v}_t^i = \frac{\kappa}{N}\sum_{j=1}^N \psi(|\hat{x}_t^j-\hat{x}_t^i|) (\hat{v}_t^j-\hat{v}_t^i)dt -\hat{x}_t^idt+\sqrt{2\sigma} \hat{v}_t^i dW_t.
\end{split}\right. %\quad i=1,\dots,N.
\end{equation}

 In the next subsection, we will show that the  stochastic system \eqref{SCS1} flocks  in the sense of Definition \ref{def}.  Since
   \begin{equation*} 
        \lim_{t \to \infty} (|x_t^i-x_t^j|+ |v_t^i-v_t^j|)= \lim_{t \to \infty} (|\hat{x}_t^i-\hat{x}_t^j|+ |\hat{v}_t^i-\hat{v}_t^j|),
        \end{equation*}
 it is sufficient to show that the  flocking  and concentration occur in the microscopic system \eqref{micro}.
  \subsection{The dynamics of the microscopic system}  In this subsection, we consider the dynamics of the microscopic variable given by system \eqref{micro}. We rewrite  it without the hat notation:
\begin{align}
	\left\{\begin{aligned}\label{1}
		&dx_t^i=v_t^i dt, \quad t>0,\\
		&dv_t^i = \frac{\kappa}{N}\sum_{j=1}^N \psi(|x_t^j-x_t^i|) (v_t^j-v_t^i) dt-x_t^idt+\sqrt{2\sigma} v_t^i dW(t),\\
		&\sum\limits_{i=1}^N x_t^i=0,\quad \sum\limits_{i=1}^N v_t^i=0. 
	\end{aligned}\right.
\end{align}
 We analyze this radially symmetric communication weight function with multiplicative noise system following the Lyapunov functional method in book \cite{M}.
   Generally speaking, to deal with the following general SDE:
  \begin{equation}\label{SDE}
  	d\mathbf{x}(t)=f(\mathbf{x}(t),t)dt+g(\mathbf{x}(t),t)dW(t), \quad \mathbf{x} = (x_1,x_2,\cdots,x_n),
  \end{equation}
   one could define the differential operator $\mathscr{L}$ associated with  equation  \eqref{SDE} by
   $$\mathscr{L}:=\frac{\partial}{\partial t}+\sum\limits_{i=1}^N f_i(\mathbf{x},t)\frac{\partial}{\partial x_i}+\frac{1}{2}\sum\limits_{i,j=1}^N[g(\mathbf{x},t)g^{T}(\mathbf{x},t)]_{ij}\frac{\partial^2}{\partial x_i\partial x_j}$$
    and act this operator on a suitable positive-definite function $V$ to derive the stochastic stability.
     Now we state the first main result.
\begin{theorem} \label{dthm}
	Let $(x_t^i, v_t^i)$ be the solution to the system \eqref{1} with bounded initial configurations and the communication weight function satisfies the following positivity and boundedness conditions 
	\begin{align*}
		\begin{aligned}
			0 < \psi_m \leq	\psi(s) \leq \psi_M, \quad s\geq0.
		\end{aligned} 
	\end{align*} 
	We assume that $\kappa\psi_m>\sigma$ and $0<\beta<\min\left\{ 1,\frac{2\kappa\psi_m-2\sigma}{1+\kappa^2\psi_M^2 }\right\}$, then system \eqref{1} flocks in velocity and concentrates in position: For $1\le i,j\le N$,
	   \begin{equation*} 
	  \lim_{t \to \infty} (|x_t^i-x_t^j|+ |v_t^i-v_t^j|)\leq  \lim_{t \to \infty}e^{-\frac{1}{3}at}=0, \quad a.s.,
	          \end{equation*}	
 where $a=\min\left\{2\kappa\psi_m-2\sigma -(1+\kappa^2\psi_M^2)\beta,\frac{\beta}{2}\right\}$.
\end{theorem}
\begin{proof}
	We define a stochastic Lyapunov functional
	\begin{align*}
		\begin{split}
			&V(\mathbf{x},t):=\alpha \sum\limits_{i=1}^{N} |x_t^i|^2+ \beta \sum\limits_{i=1}^{N} x_t^i\cdot v_t^i+\sum\limits_{i=1}^{N}  |v_t^i|^2,  \\
		\end{split}
	\end{align*}	
where $\alpha$ and $\beta$ are positive constants to be determined later and $\mathbf{x} = (x_t^1,x_t^2,\cdots,x_t^N,v_t^1,v_t^2,\cdots,v_t^N)$.
Then we get 
\begin{align*}%\label{2}
	\begin{split}
		& V_t(\mathbf{x},t)=0,  \\
		& V_\mathbf{x}(\mathbf{x},t)=(2\alpha x_t^1+ \beta v_t^1, \cdots ,2\alpha x_t^N+ \beta v_t^N, \beta x_t^1+2v_t^1, \cdots  , \beta x_t^N+2v_t^N),\\
    \end{split}
\end{align*}
		$\qquad \qquad \quad V_{\mathbf{x}\mathbf{x}}(\mathbf{x},t)=\begin{pmatrix} 
		2\alpha & 0 &\cdots&0&\beta&0&\cdots&0 \\ 
		0&2\alpha &\cdots&0&0&\beta&\cdots&0\\ 
		\vdots&\vdots &\ddots&\vdots&\vdots &\vdots &\ddots &\vdots\\
		0&0 &\cdots &2\alpha &0 &0&\cdots &\beta\\
		\beta & 0 &\cdots&0&2&0&\cdots&0 \\ 
		0&\beta &\cdots&0&0&2&\cdots&0\\ 
		\vdots&\vdots &\ddots&\vdots&\vdots &\vdots &\ddots &\vdots\\
		0&0 &\cdots &\beta &0 &0&\cdots &2
	\end{pmatrix}.\\$ 

Now we compute %\eqref{sumv}
	\begin{align*}
		\begin{aligned} %\label{sumv}
			\mathscr{L}[V](\mathbf{x},t)=&V_t(\mathbf{x},t)+V_\mathbf{x}(\mathbf{x},t)f(\mathbf{x},t)+\frac{1}{2}trace[g^T(\mathbf{x},t)V_{\mathbf{x}\mathbf{x}}g(\mathbf{x},t)],\\
		\end{aligned}
	\end{align*}
 where 
\begin{align*}
    \begin{aligned} %\label{a}
 	&f(\mathbf{x},t)=\!(v_t^1,\cdots,v_t^N,\frac{\kappa}{N}\!\sum_{j=1}^N \psi(|x_t^j\!-\!x_t^1|) (v_t^j\!-\!v_t^1) \!-\!x_t^1,\cdots,\frac{\kappa}{N}\!\sum_{j=1}^N \psi(|x_t^j\!-\!x_t^N|) (v_t^j\!-\!v_t^N)\!-\!x_t^N)^T,\\
 	&g(\mathbf{x},t)=(0,0,\cdots,0, \sqrt{2\sigma} v_t^1, \sqrt{2\sigma} v_t^2, \cdots,\sqrt{2\sigma}v_t^N)^T.
	\end{aligned}
\end{align*}
  %We combine \eqref{sumv} and \eqref{a}  to  get
 We have
  	\begin{align}
	\begin{aligned} \label{sum1}
		\mathscr{L}[V](\mathbf{x},t)=& \frac{2\kappa}{N} \sum_{i=1}^N \sum_{j=1}^N \psi(|x_t^j-x_t^i|)(v_t^j-v_t^i)\cdot v_t^i+\frac{\kappa \beta}{N} \sum_{i=1}^N \sum_{j=1}^N \psi(|x_t^j-x_t^i|)(v_t^j-v_t^i)\cdot x_t^i\\
		&+(2\alpha-2) \sum\limits_{i=1}^{N}x_t^i\cdot v_t^i+(\beta+2\sigma )\sum\limits_{i=1}^{N}|v_t^i|^2-\beta\sum\limits_{i=1}^{N}|x_t^i|^2.\\
	\end{aligned}
\end{align}
	%$g^TV_{xx}g=(0,\cdots,0,Dv_1,\cdots,Dv_N)V_{xx}(0,\cdots,0,Dv_1,\cdots,Dv_N)^T=2\sum\limits_{i=1}^{N}D^2v_i^2$
	As $\psi$ is a symmetric function, the first term on the right-hand side of \eqref{sum1} can be treated by exchanging $i\leftrightarrow j$:
	\begin{align*}
		\begin{aligned} 
			&\frac{2\kappa}{N} \sum_{i=1}^N \sum_{j=1}^N \psi(|x_t^j-x_t^i|)(v_t^j-v_t^i)\cdot v_t^i\\
			%=&-\frac{2\kappa}{N} \sum_{i=1}^N \sum_{j=1}^N \psi(|x_t^j-x_t^i|)(v_t^j-v_t^i)\cdot v_t^j\\
			=&-\frac{\kappa}{N}\sum\limits_{i,j=1}^N   \psi(|x_t^j-x_t^i|) |v_t^j-v_t^i|^2\\
			\le&- \frac{\kappa}{N}\psi_m \sum\limits_{i,j=1}^N |v_t^j-v_t^i| ^2\\
			=&- 2\kappa\psi_m \sum_{i=1}^N |v_t^i|^2,
		\end{aligned}
	\end{align*}
where we used $\sum\limits_{i=1}^N v_t^i=0$.
The second term on the right-hand side of \eqref{sum1} can be written:
	\begin{align*}
		\begin{aligned} 
			&\frac{\kappa\beta}{N} \sum_{i=1}^N \sum_{j=1}^N \psi(|x_t^j-x_t^i|)(v_t^j-v_t^i)\cdot x_t^i\\
			\le&\frac{\kappa\beta}{N}\sum\limits_{i,j=1}^N   \psi(|x_t^j-x_t^i|) |(v_t^j-v_t^i)\cdot x_t^i|\\
			\le&\frac{\kappa\beta}{N} \psi_M\sum\limits_{i,j=1}^N  \left[\frac{\kappa\psi_M}{2}|v_t^j-v_t^i|^2+\frac{|x_t^i|^2}{2\kappa\psi_M}\right]\\
			=& \kappa^2\psi_M^2\beta\sum\limits_{i=1}^N  |v_t^i|^2+\frac{\beta}{2}\sum\limits_{i=1}^N |x_t^i|^2,
		\end{aligned}
	\end{align*}
where we used $|(v_t^j-v_t^i)\cdot x_t^i|\le \left(\kappa\psi_M|v_t^j-v_t^i|^2+\frac{|x_t^i|^2}{\kappa\psi_M}\right)/2$.

	Therefore  we get the following estimate for $\mathscr{L}[V](\mathbf{x},t)$: 
	\begin{align*}
		\begin{aligned} 
			\mathscr{L}[V](\mathbf{x},t)
			\le&(2\alpha-2) \sum\limits_{i=1}^{N}x_t^iv_t^i-( 2\kappa\psi_m-2\sigma-(1+\kappa^2\psi_M^2)\beta)\sum\limits_{i=1}^{N}|v_t^i|^2-\frac{\beta}{2}\sum\limits_{i=1}^{N}|x_t^i|^2.
		\end{aligned}
	\end{align*}
	In order to convert $\mathscr{L}[V]$ to be negative-definite, we set $2\alpha-2=0,$ that is $\alpha=1.$	
	Then $\mathscr{L}[V]$ and $V$ become
	\begin{equation*}
		\mathscr{L}[V]\le-( 2\kappa\psi_m-2\sigma -(1+\kappa^2\psi_M^2)\beta)\sum\limits_{i=1}^{N}|v_t^i|^2-\frac{\beta}{2}\sum\limits_{i=1}^{N}|x_t^i|^2 \quad
	\end{equation*}
	and
	\begin{equation*}
		V(\mathbf{x},t)= \sum\limits_{i=1}^{N}|x_t^i|^2+ \beta \sum\limits_{i=1}^{N} x_t^i\cdot v_t^i+ \sum\limits_{i=1}^{N}|v_t^i|^2.
	\end{equation*}\\
	%For  $\mathscr{L}[V](\mathbf{x},t)$ to be negative-definite, we must make the coefficient in front of the square term less than 0, that is
	In order to make $\mathscr{L}[V]$ negative-definite, we set
	%$\beta\!>\!0$ and	$2\kappa\psi_m-2\sigma -(1+\!\kappa^2\psi_M^2)\beta\!>\!0$,
	%that is
	\begin{align*}
		\begin{aligned} 
			&0<\beta<\frac{2\kappa\psi_m-2\sigma}{1+\kappa^2\psi_M^2}.
		\end{aligned}
	\end{align*}
	%For  $V(\mathbf{x},t)$ to be positive-definite, we must have $\Delta=\beta^2-4<0,$ that is
%	\begin{align}
%		\begin{aligned} \label{b2}
%			&-2<\beta<2.
%		\end{aligned}
%	\end{align} 
%Combining \eqref{b1} and \eqref{b2}, we see that if 
%	\begin{align*}
%		\begin{aligned} 
%			&0<\beta<\min\left\{ 2,\frac{2\kappa\psi_m-2\sigma}{1+\kappa^2\psi_M^2 }\right\},
%		\end{aligned}
%	\end{align*}
%	then $V$ is positive-definite and $\mathscr{L}[V]$ is negative-define. \\

Since $|\beta|\le1$, the quantity $\sum\limits_{i=1}^{N}|x_t^i|^2+ \beta \sum\limits_{i=1}^{N} x_t^i\cdot v_t^i+ \sum\limits_{i=1}^{N}|v_t^i|^2$ is equivalent to $\sum\limits_{i=1}^{N}|x_t^i|^2+ \sum\limits_{i=1}^{N}|v_t^i|^2$:
\[\frac{3}{8} \left(\sum\limits_{i=1}^{N}|x_t^i|^2+\sum\limits_{i=1}^{N}|v_t^i|^2\right) \leq \sum\limits_{i=1}^{N}|x_t^i|^2+ \beta \sum\limits_{i=1}^{N} x_t^i\cdot v_t^i+ \sum\limits_{i=1}^{N}|v_t^i|^2 \leq \frac{3}{2}\left(\sum\limits_{i=1}^{N}|x_t^i|^2+\sum\limits_{i=1}^{N}|v_t^i|^2\right) .\]
 Then we have 
\begin{align*}
	\begin{aligned} 
	\mathscr{L}[V] \leq -a\left(\sum\limits_{i=1}^{N}|x_t^i|^2+  \sum\limits_{i=1}^{N}|v_t^i|^2\right) \leq -\frac{2a}{3}V,
	\end{aligned}
\end{align*}
where $a=\min\left\{2\kappa\psi_m-2\sigma -(1+\kappa^2\psi_M^2)\beta,\frac{\beta}{2}\right\}$.

By Corollary 3.4 in  book \cite{M},  we conclude that if  
	\begin{equation*}
		\kappa\psi_m>\sigma \quad and \quad 0<\beta<\min\left\{ 1,\frac{2\kappa\psi_m-2\sigma}{1+\kappa^2\psi_M^2 }\right\},
	\end{equation*}
	then we have
	\begin{equation*}
		\lim_{t \to \infty} \sup\frac{1}{t}\ln\left(\sum_{i=1}^N|x_t^i|+\sum_{i=1}^N|v_t^i|\right)\le -\frac{1}{3}a<0,\quad a.s.
	\end{equation*}
	Thus, we infer 
		%\begin{align*}
		%\begin{aligned}%\label{q} 
		%	&\lim_{t \to \infty}  |x_t^i|=  0,\quad a.s,\\
			%&\lim_{t \to \infty}  |v_t^i|=  0,\quad a.s.
		%\end{aligned}
	%\end{align*}
	  \begin{equation*} 
        \lim_{t \to \infty} (|x_t^i-x_t^j|+ |v_t^i-v_t^j|) \le  \lim_{t \to \infty} \left(\sum_{i=1}^N|x_t^i|+\sum_{i=1}^N|v_t^i|\right)\leq  \lim_{t \to \infty}e^{-\frac{1}{3}at}=0, \quad a.s.,
        \end{equation*}
          i.e., the system  \eqref{1} flocks and concentrates.
\end{proof}

\subsection{Numerical results for the microscopic system}
In this subsection, we give some numerical results  for the  microscopic dynamics  \eqref{1}  with two kinds of communication weight functions, and compare them with the analytic results in the last  subsection.

Firstly, we employ a constant communication weight function $\psi(r)=1$ and the parameters are $\kappa=100$, $  \sigma=200$. We solve the system for 100 particles by using Euler's method in two dimensional space.  The initial locations and velocities for system are randomly distributed in the interval [-50, 50]$\times$[-50, 50] and satisfy the conservation laws $\eqref{1}_3$. The result is shown in figure \ref{fig1}. 

Secondly, we employ the communication weight function \[\psi(r)=\frac{1}{(1+r^2)^{1/4}}.\]
We select initial configuration for 100 particles in the same way as in the constant communication weight function example. The other parameters are set as before, $\kappa=100$, $\sigma=200$. Figure \ref{fig2} shows all the realizations of the trajectories of $v^1$ and $x^1$.

\begin{figure}[htbp] 
	\centering
	\subfigure[Trajectories of the velocity $v^1$ for 100 particles.] 
	{\includegraphics[width=6cm]{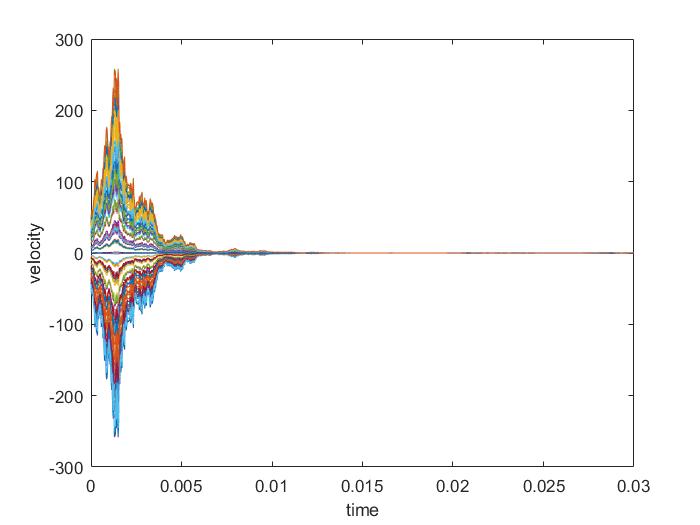}
	}
	%\quad
	%\subfigure[All the 100 realizations of the trajectories of  $v^1_1$. The average of  all the realizations is the solid line.] 
	%{\includegraphics[width=6cm]{v2.jpg}
	%}
	\subfigure[Trajectories of the position $x^1$ for 100 particles.] 
	{\includegraphics[width=6cm]{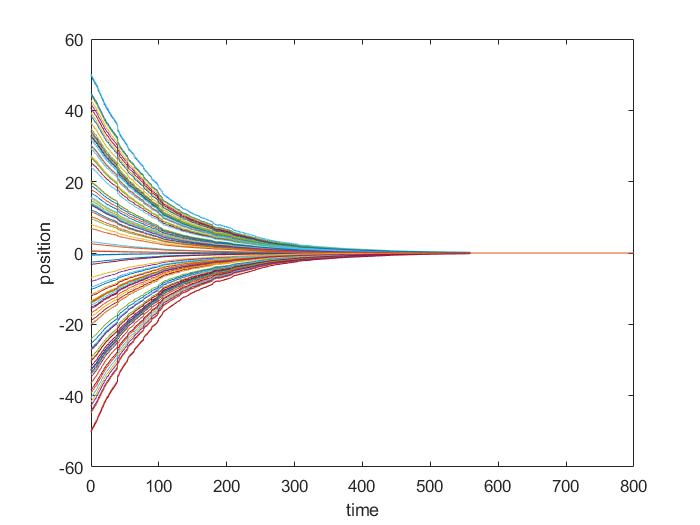}}
	%\quad
	%\subfigure[All the 100 realizations of the trajectories of  $x^1_1$. The average of  all the realizations is the solid line.] 
	%{\includegraphics[width=6cm]{x2.jpg}}
	\caption{ The  constant  communication weight function.}\label{fig1}
\end{figure}
\begin{figure}[h!] 
	\centering
	\subfigure[Trajectories of the velocity $v^1$ for 100 particles.] 
	{\includegraphics[width=6cm]{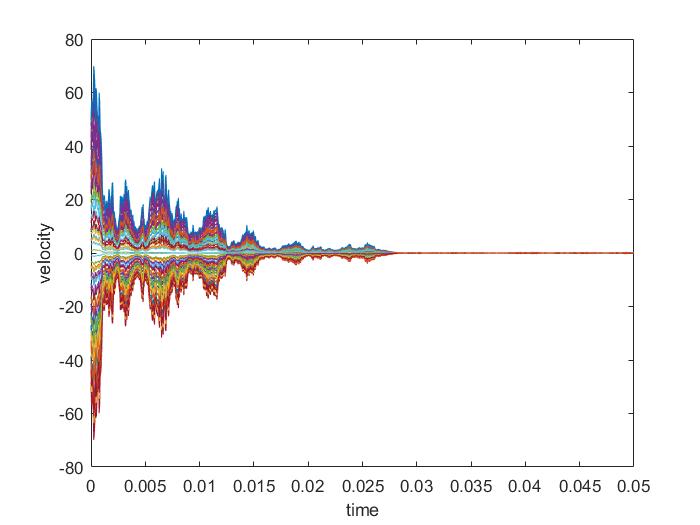}
	}
	\subfigure[Trajectories of the position $x^1$ for 100 particles.] 
{\includegraphics[width=6cm]{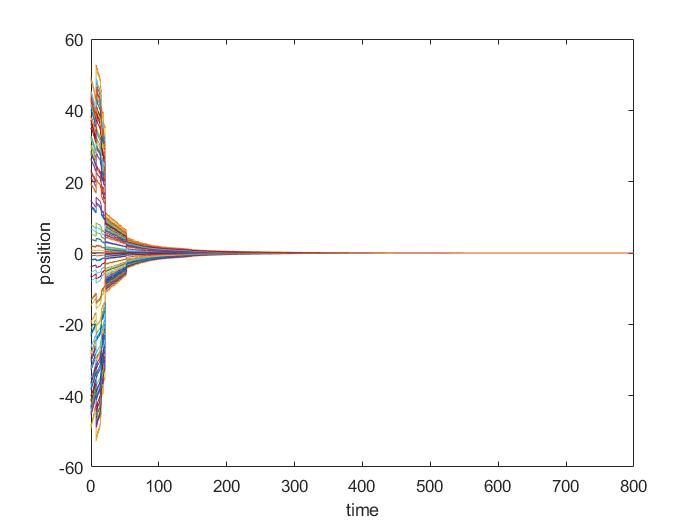}
}

\caption{The radially symmetric communication weight  function.}\label{fig2}
\end{figure}
From figure \ref{fig1} and figure \ref{fig2}, we find that as time goes to infinity, all particles move at the same position and velocity. Then we confirm  Theorem \ref{dthm}.

%\newpage

\section{Kinetic C-S Vlasov-type equation}\label{sec:3}
\setcounter{equation}{0}
  In this section, we consider the  existence  and asymptotic long-time behavior of solutions to the kinetic C-S Vlasov-type model in a harmonic field. 

\subsection{Existence of the global classical solution}
 % {\bf Notation.} 
   Similar to \cite{J, HKR}, the global existence of classical solution to kinetic  C-S Vlasov-type model with noise can be established by introducing  a new weighted  Sobolev space.
  We state the main result without the proof.
  For a measurable function $f(x,v,t)$ in the phase space $\mathbb{R}^{2d}$, we set
  \begin{align*}
			\begin{aligned}
				&\|f\|_{L^1(x,v)}=\|(1+|x|^2+|v|^2)f\|_{L^1(\mathbb{R}^{2d})},\\ &\|f\|_{L_{\alpha,x,v}^2}=\int_{\mathbb{R}^{2d}}(1+|x|^2+|v|^2)^\alpha |f|^2 dxdv,\quad \alpha\ge 0, \\
				&\|f\|^2_{H_\alpha^k}:= \|f\|_{L_{\alpha,x,v}^2}+\sum\limits_{1\le i+j\le k}\|\partial_x^i\partial_v^j f\|_{L^2}^2,\quad k\in \mathbb{N}\cup\{0\}.
			\end{aligned} 
		\end{align*}

    We define a function space where we will look for a classical solution: for $T>0$,
   \[ \mathcal{X}_{k,\alpha}(T):=\{ f\in C(0,T; (H_\alpha^k\cap L^1_{(x,v)})(\mathbb{R}^{2d})):\sup\limits_{t\in[0,T)}(\|f(t)\|_{H_\alpha^k}+\|f(t)\|_{L^1_{(x,v)}})<\infty\}.\]
  \begin{theorem}
   Let $T\in(0,\infty)$ be a positive constant and we assume that the initial configuration $f_0$ satisfies 
   \begin{equation*}
   f_0\in H_{\alpha}^k\cap L^1_{(x,v)}, \ \  \mbox{for some positive constants}\ \  k>2+d,\ \ \alpha>\frac{d+2}{2}.
   \end{equation*}
    Then, there exists a unique global classical solution to the Cauchy problem of  \eqref{km}-\eqref{k1}  in the function space  $\mathcal{X}_{k,\alpha}(T)$.
  \end{theorem}
    \subsection{Flocking  and concentration  behavior for the kinetic model}
     In this subsection, we  derive a sufficient condition leading to a  exponential  flocking  and concentration results for the kinetic  model.  Firstly, we present  a lemma:
\begin{lemma}\label{le}
	Let $f=f(x,v,t)$ be a smooth solution to \eqref{km}-\eqref{k1} that quickly decays to zero at infinity. Then we have the following estimates:
	\begin{enumerate}[(1)]
		\item 	The total mass is conserved:
		\begin{align*}
			\begin{aligned}
				\frac{d}{dt}\int_{\mathbb{R}^{2d}} f dxdv = 0, \quad t>0.
			\end{aligned} 
		\end{align*}
		%W.o.l.g., we  can assume that $\int_{\mathbb{R}^{2d}} f dxdv =1$ in the rest of this paper.
		\item 
		 We define the means $x_c=\int_{\mathbb{R}^{2d}} xfdxdv/\int_{\mathbb{R}^{2d}} f dxdv$ and $v_c=\int_{\mathbb{R}^{2d}} vfdxdv/\int_{\mathbb{R}^{2d}} f dxdv$. 
They are governed by the harmonic oscillators:
		\begin{equation}\label{macro}
			\left\{\begin{split}
				& dx_c = v_cdt,\\
				& dv_c = - x_c dt.
			\end{split}\right.
		\end{equation}
	\end{enumerate} 	
\end{lemma}
\begin{proof}
	Note that the equation in \eqref{km} can be rewritten in a divergent form: 
	\begin{align}\label{l1}
		\begin{split}
			&\partial_t f +  \nabla_{x} \cdot (vf) + \nabla_{v} \cdot (L[f]f-  xf-\sigma \nabla_{v}(|v-v_c|^2f))=0 .   \\
		\end{split}
	\end{align}
We integrate \eqref{l1} in $(x,v)$ to obtain the conservation of mass.

W.o.l.g., we  can assume that $\int_{\mathbb{R}^{2d}} f dxdv =1$ here. By the definition of $x_c$ and $v_c$, we have
	\begin{align*}
		\begin{split}
			&\frac{dx_c}{dt} = \int_{\mathbb{R}^{2d}} xf_tdxdv, \\
			&\frac{dv_c}{dt} = \int_{\mathbb{R}^{2d}} vf_tdxdv.
		\end{split}
	\end{align*}
We multiply $x$ and $v$ with \eqref{l1} respectively to get
	\begin{align}
		\begin{split}\label{r1}
			&\partial_t (xf) +  \nabla_{x} \cdot (x \otimes vf) - vf + \nabla_{v} \cdot \left[x\otimes \left(L[f]f-  xf-\sigma \nabla_{v}(|v-v_c|^2f)\right)\right]=0 
		\end{split}
	\end{align}
and
\begin{align}
	\begin{split}\label{r2}
		&\partial_t (vf) +\! \nabla_{x} \!\cdot (v\otimes v f) + \nabla_{v} \cdot\left [v \otimes L[f]f\! -\sigma v\otimes \nabla_{v}(|v-v_c|^2f)-x\otimes vf\right]+ xf\!\\
		&\quad \quad \quad \quad \qquad \quad \quad \quad \quad 
		=L[f]f-\sigma\nabla_{v}(|v-v_c|^2f). 
	\end{split}
\end{align}
Then we integrate the relation in \eqref{r1}-\eqref{r2} and use
\begin{align*}
	\begin{split}
		&\int_{\mathbb{R}^{2d}}L[f]fdvdx = -\kappa \int_{\mathbb{R}^{4d}}\psi(|x-y|)(v-v_*)f(y,v_*,t)f(x,v,t)dvdv_*dydx=0
	\end{split}
\end{align*}
to obtain \eqref{macro}.
\end{proof}
 
To prove a flocking and concentration estimate for the kinetic model, we introduce a Lyapunov functional
	\begin{equation}\label{L}
		\tilde{\mathcal L}[f](t)
		:=\int_{\mathbb{R}^{2d}}\left(\frac{1}{2}|x-x_c|^2+\frac{1}{2}|v-v_c|^2+\epsilon (x-x_c)\cdot(v-v_c)\right)fdxdv, \quad |\epsilon|\le \frac{1}{2}.
	\end{equation} 
It is equivalent to the standard  Lyapunov variance functional measuring the  position and velocity  variances of the kinetic density function $f$:
 \begin{equation}\label{L1}
		\mathcal L[f](t)
		:=\int_{\mathbb{R}^{2d}}\left(|x-x_c|^2+|v-v_c|^2\right)fdxdv, 
	\end{equation} 
i.e., $\frac{3}{16}\mathcal L[f](t)\le\tilde{\mathcal L}[f](t)\le\frac{3}{4}\mathcal L[f](t).$
\begin{theorem} \label{u4}
	Suppose that the communication weight function $\psi$ satisfies the following positivity and boundedness conditions:  there exist positive constants $\psi_M$ and $\psi_m$ such that 
	\begin{align*}
		\begin{aligned}
		0 < \psi_m \leq	\psi(s) \leq \psi_M, \quad s\geq0,
		\end{aligned} 
	\end{align*} 
	and let $f=f(x,v,t)$ be a classical solution to \eqref{km} that quickly decays to zero at infinity and satisfies the finite second moments
	\begin{align*}
		\begin{aligned}
			%{\color{red}
			\int_{\mathbb{R}^{2d}}(1+|v-v_c|^2+|x-x_c|^2)f(x,v,t)dxdv<\infty, \quad t\geq0.%}
		\end{aligned} 
	\end{align*} 
	Then, the following estimates hold:\\
	\begin{enumerate}[(1)]
		\item 	If $\kappa>\frac{d\sigma}{\psi_m||f_0||_{L^1}}$, there exists a positive constant $C_m:=\min \left\{ \frac{1}{4},\frac{\kappa\psi_m||f_0||_{L^1}-d\sigma}{4(1+2(\kappa\psi_M)^2)} \right\}$, 
		%$\epsilon=\min \{\frac{1}{2},\frac{\kappa\psi_m||f_0||_{L^1}-d\sigma}{2(1+2(\kappa\psi_M)^2)}\}$,
		such that
		\begin{align}  \label{mkf}
			\begin{aligned} 
				&\mathcal L[f](t) \leq 4\mathcal L[f_0]e^{-\frac{4}{3}C_mt} .
			\end{aligned} 
		\end{align}
		\item 	If $\kappa<\frac{d\sigma}{\psi_M||f_0||_{L^1}}$, there exists a positive constant $C_M: = \min \left\{\frac{1+2(\kappa \psi_M)^2}{2} ,\frac{d\sigma-\kappa\psi_M||f_0||_{L^1}}{2} \right\}$,
	%	$\epsilon=\max \{-\frac{1}{2},\kappa\psi_M||f_0||_{L^1}-d\sigma\}$,
		such that
		\begin{align} \label{nf}
			\begin{aligned}
				&\mathcal L[f](t) \geq\frac{1}{4}\mathcal L[f_0]e^{\frac{4}{3}C_Mt} .
			\end{aligned} 
		\end{align}
	\end{enumerate} 
\end{theorem}
\begin{proof}
		We estimate  each component in  \eqref{L}:\newline
	\noindent $\bullet$ (Estimate for $\int_{\mathbb{R}^{2d}}|x-x_c|^2fdxdv$):~
	%By multiplying $|x-x_c|^2$ with \eqref{km}, we get
	%\begin{align}
	%	\begin{aligned}  \label{S1}
	%		&\partial_t(|x-x_c|^2f) +  \nabla_{x}  \cdot (v|x-x_c|^2f)\\
	%		& \quad+ \nabla_{v} \cdot |x-x_c|^2 (L[f]f -xf-\sigma \nabla_{v} (|v-v_c|^2f)) \\
	%		&\quad =2(x-x_c) \cdot vf.
	%	\end{aligned} 
	%\end{align} 
	By straightforward calculations, we get
	\begin{align}
		\begin{aligned}  \label{1.1}
			\frac{d}{dt}\left( \int_{\mathbb{R}^{2d}}|x-x_c|^2fdxdv\right)&=-2\int_{\mathbb{R}^{2d}}(x-x_c)\cdot \frac{dx_c}{dt}fdxdv+\int_{\mathbb{R}^{2d}}|x-x_c|^2f_tdxdv\\
			&= \int_{\mathbb{R}^{2d}}|x-x_c|^2f_tdxdv,\\
		\end{aligned} 
	\end{align}
	where the first term in \eqref{1.1} vanished by the definition of $x_c$ in Lemma \ref{le}. Therefore \eqref{1.1} can be estimated as follows:
	\begin{align}\label{l}
		\begin{aligned}
			\frac{d}{dt}\left( \int_{\mathbb{R}^{2d}}|x-x_c|^2fdxdv\right)&=\int_{\mathbb{R}^{2d}}|x-x_c|^2f_tdxdv=2\int_{\mathbb{R}^{2d}}(x-x_c) \cdot vfdxdv\\
			& =2\int_{\mathbb{R}^{2d}}(x-x_c) \cdot (v-v_c)fdxdv,
		\end{aligned}
	\end{align} 
	where we used $\int_{\mathbb{R}^{2d}}(x-x_c)fdxdv=0$.
	\newline 
	\noindent $\bullet$ (Estimate for $\int_{\mathbb{R}^{2d}}|v-v_c|^2fdxdv$):~We calculate the derivative of it in a direct way to obtain
	%multiply $|v-v_c|^2$ with \eqref{km} to obtain
	%\begin{align}
	%	\begin{aligned}  \label{S2}
	%		&\partial_t(|v-v_c|^2f) +  \nabla_{x}  \cdot (v|v-v_c|^2f)\\
	%		&\quad + \nabla_{v} \cdot (|v-v_c|^2 L[f]f -|v-v_c|^2xf\\
	%		&\quad-\sigma |v-v_c|^2\nabla_{v} (|v-v_c|^2f) +2\sigma(v-v_c)|v-v_c|^2f )\\
	%		&\quad=2(v-v_c) \cdot (L[f]f)+2d\sigma|v-v_c|^2f-2(v-v_c) \cdot xf.
	%	\end{aligned} 
	%\end{align} 
	\begin{align}
		\begin{aligned} \label{1.2}
			&\frac{d}{dt}\left( \int_{\mathbb{R}^{2d}}|v-v_c|^2fdxdv\right)= \int_{\mathbb{R}^{2d}}|v-v_c|^2f_tdxdv\\
			&=2\!\!\int_{\mathbb{R}^{2d}}\!(v\!-\!v_c)\cdot (L[f]f)dxdv\!+\!2d\sigma\!\!\int_{\mathbb{R}^{2d}}\!|v\!-\!v_c|^2fdxdv\!-\!2\int_{\mathbb{R}^{2d}}\!(x\!-\!x_c)\! \cdot\! (v\!-\!v_c)fdxdv,
		\end{aligned} 
	\end{align} 
where we used $\int_{\mathbb{R}^{2d}}(v-v_c)fdxdv=0$. 
Note that the first term on the R. H. S. of \eqref{1.2} can be estimated as follows:
	\begin{align}
		\begin{aligned}  \label{1.3}
			&\int_{\mathbb{R}^{2d}}(v-v_c) \cdot (L[f]f)dxdv\\
			&=-\kappa\int_{\mathbb{R}^{4d}}\psi(|x-y|)(v-v_c) \cdot (v-v_\ast)f(y,v_\ast,t)f(x,v,t)dv_\ast dvdydx\\
			&=\kappa\int_{\mathbb{R}^{4d}}\psi(|x-y|)(v_\ast-v_c) \cdot (v-v_\ast)f(y,v_\ast,t)f(x,v,t)dv_\ast dvdydx\\
			&=-\frac{\kappa}{2}\int_{\mathbb{R}^{4d}}\psi(|x-y|)|v-v_\ast|^2 f(y,v_\ast,t)f(x,v,t)dv_\ast dvdydx.\\
		\end{aligned} 
	\end{align} 
	Then, we combine  \eqref{1.2} and  \eqref{1.3} to obtain:
	\begin{align}
		\begin{aligned}  \label{1.4}
			&\frac{d}{dt}\left( \int_{\mathbb{R}^{2d}}|v-v_c|^2fdxdv\right)\\
			& =-\kappa\int_{\mathbb{R}^{4d}}\psi(|x-y|)|v-v_\ast|^2 f(y,v_\ast,t)f(x,v,t)dv_\ast dvdydx\\
			&\quad +2d\sigma\int_{\mathbb{R}^{2d}}|v-v_c|^2fdxdv-2\int_{\mathbb{R}^{2d}}(x-x_c) \cdot (v-v_c)fdxdv.
		\end{aligned} 
	\end{align} 
	\noindent $\bullet$ (Estimate for $\int_{\mathbb{R}^{2d}}(x-x_c)\cdot (v-v_c)fdxdv$):~We calculate the cross term directly:
	%\begin{align}
	%	\begin{aligned}   \label{1.5}
	%		&\partial_t((x-x_c) \cdot (v-v_c)f) +  \nabla_{x}  \cdot ((x-x_c)\cdot(v-v_c) vf)\\
	%		&\quad + \nabla_{v} \cdot ((x-x_c)\cdot(v-v_c) L[f]f -(x-x_c)\cdot(v-v_c) xf+\sigma(x-x_c)|v-v_c|^2f )\\
	%		&\quad- \sigma \nabla_{v} \cdot [(x-x_c) \cdot (v-v_c)\nabla_{v}(|v-v_c|^2f)]\\
	%		&\quad=(v-v_c)\cdot vf-(x-x_c)\cdot xf+(x-x_c)\cdot(L[f]f).
	%	\end{aligned} 
	%\end{align} 
	\begin{align}
		\begin{aligned}   \label{1.6}
			&\frac{d}{dt}\left( \int_{\mathbb{R}^{2d}}(x-x_c)\cdot(v-v_c)fdxdv\right)= \int_{\mathbb{R}^{2d}}(x-x_c)\cdot(v-v_c)f_tdxdv\\
			&=\int_{\mathbb{R}^{2d}}|v-v_c|^2fdxdv -\int_{\mathbb{R}^{2d}}|x-x_c|^2fdxdv+\int_{\mathbb{R}^{2d}}(x-x_c)\cdot L[f]fdxdv,
		\end{aligned} 
	\end{align} 
	where we used $\int_{\mathbb{R}^{2d}}(x-x_c)fdxdv=0$ and $\int_{\mathbb{R}^{2d}}(v-v_c)fdxdv=0$.
	The third term on the R.H.S. of \eqref{1.6} can be estimated as follows:
	\begin{align}
		\begin{aligned}  \label{1.7}
			&\int_{\mathbb{R}^{2d}}(x-x_c) \cdot (L[f]f)dxdv\\
			& =-\kappa\int_{\mathbb{R}^{4d}}\psi(|x-y|)(x-x_c) \cdot (v-v_\ast)f(y,v_\ast,t)f(x,v,t)dv_\ast dvdydx\\
			& =-\kappa\int_{\mathbb{R}^{4d}}\psi(|x-y|)(x-x_c) \cdot (v-v_c)f(y,v_\ast,t)f(x,v,t)dv_\ast dvdydx\\
			&\quad-\kappa\int_{\mathbb{R}^{4d}}\psi(|x-y|)(x-x_c) \cdot (v_c-v_\ast)f(y,v_\ast,t)f(x,v,t)dv_\ast dvdydx\\
			& \leq \kappa \psi_M\int_{\mathbb{R}^{4d}}\left(\frac{|x-x_c|^2}{4\kappa\psi_M}  + \kappa \psi_M|v-v_c|^2\right) f(y,v_\ast,t)f(x,v,t)dv_\ast dvdydx\\
			&\quad+ \kappa \psi_M\int_{\mathbb{R}^{4d}}\left(\frac{|x-x_c|^2}{4\kappa\psi_M}  + \kappa\psi_M|v_c-v_\ast|^2\right)f(y,v_\ast,t)f(x,v,t)dv_\ast dvdydx\\
			&=\frac{1}{2}\int_{\mathbb{R}^{2d}}|x-x_c|^2f(x,v,t)dxdv+ 2(\kappa\psi_M)^2\int_{\mathbb{R}^{2d}}|v-v_c|^2f(x,v,t)dxdv,
		\end{aligned} 
	\end{align} 
	where we used $(x-x_c) \cdot (v-v_c) \leq \frac{|x-x_c|^2}{4\kappa\psi_M}  + \kappa\psi_M|v-v_c|^2 $ and $(x-x_c) \cdot (v_c-v_\ast) \leq \frac{|x-x_c|^2}{4\kappa\psi_M}  + \kappa\psi_M|v_c-v_\ast|^2 $.
	Then, we combine  \eqref{1.6} and  \eqref{1.7} to obtain an estimate for the cross term:
	\begin{align}
		\begin{aligned}   \label{1.8}
			&\frac{d}{dt}\left( \int_{\mathbb{R}^{2d}}(x-x_c)\cdot(v-v_c)fdxdv\right)\\
			&\leq (1+2(\kappa\psi_M)^2)\int_{\mathbb{R}^{2d}}|v-v_c|^2fdxdv -\frac{1}{2}\int_{\mathbb{R}^{2d}}|x-x_c|^2fdxdv.
		\end{aligned} 
	\end{align}

	Finally, we take $\frac{1}{2} \eqref{l}+\frac{1}{2}\eqref{1.4} +\epsilon\eqref{1.8}$ to obtain a differential inequality for $\tilde{\mathcal L}[f](t)$:
	\begin{align}
		\begin{aligned} \label{1.9}
			&\frac{d\tilde{\mathcal L}[f](t)}{dt}
			=\frac{d}{dt}\int_{\mathbb{R}^{2d}}\left(\frac{1}{2}|x-x_c|^2+\frac{1}{2}|v-v_c|^2+\epsilon (x-x_c)\cdot(v-v_c)\right)fdxdv\\
			&\leq\!-(\kappa\psi_m||f_0||_{L^1}\!-\!d\sigma\!-\!(1+2(\kappa\psi_M)^2)\epsilon)\!\int_{\mathbb{R}^{2d}} |v\!-\!v_c|^2 fdxdv\!-\!\frac{\epsilon}{2}\!\int_{\mathbb{R}^{2d}} |x\!-\!x_c|^2 fdxdv,
		\end{aligned} 
	\end{align}
where we used $\int_{\mathbb{R}^{2d}}(v-v_c)fdxdv=0$ and  $|v-v_\ast|^2=|v-v_c+v_c-v_\ast|^2$.
	One can simplify \eqref{1.9} by taking $\epsilon=\min \{\frac{1}{2},\frac{\kappa\psi_m||f_0||_{L^1}-d\sigma}{2(1+2(\kappa\psi_M)^2)}\}$:
	\begin{align}
		\begin{aligned} \label{1.10}
			&\frac{d}{dt}\int_{\mathbb{R}^{2d}}\left(\frac{1}{2}|x-x_c|^2+\frac{1}{2}|v-v_c|^2+\epsilon (x-x_c) \cdot (v-v_c)\right)fdxdv\\
			&\leq -C_m \int_{\mathbb{R}^{2d}} (|x-x_c|^2+|v-v_c|^2) fdxdv ,
		\end{aligned} 
	\end{align}
	where $C_m=\min \{ \frac{\epsilon}{2},\frac{\kappa\psi_m||f_0||_{L^1}-d\sigma}{2}\}$.
	In \eqref{1.10}, we  apply the Gr\"onwall's inequality to get
	\begin{align*}
		\begin{aligned}% \label{1.11}
			&\tilde{\mathcal L}[f](t) \leq \tilde{\mathcal L}[f_0]e^{-\frac{4}{3}C_mt}
		\end{aligned} 
	\end{align*}
	and  conclude the first result \eqref{mkf}.
	
	On the other hand, we derive
	\begin{align}
		\begin{aligned}  \label{2.1}
			&\int_{\mathbb{R}^{2d}}(x-x_c) \cdot (L[f]f)dxdv\\
			&=-\kappa\int_{\mathbb{R}^{4d}}\psi(|x-y|)(x-x_c) \cdot (v-v_c)f(y,v_\ast,t)f(x,v,t)dv_\ast dvdydx\\
			&\quad -\kappa\int_{\mathbb{R}^{4d}}\psi(|x-y|)(x-x_c) \cdot (v_c-v_\ast)f(y,v_\ast,t)f(x,v,t)dv_\ast dvdydx\\
			&\geq -\kappa \psi_M\int_{\mathbb{R}^{4d}}\left[\kappa\psi_M|x-x_c|^2  + \frac{|v-v_c|^2}{4\kappa\psi_M}\right]f(y,v_\ast,t)f(x,v,t)dv_\ast dvdydx\\
			&\quad - \kappa \psi_M\int_{\mathbb{R}^{4d}}\left[\kappa\psi_M|x-x_c|^2  + \frac{|v_c-v_\ast|^2}{4\kappa\psi_M}\right] f(y,v_\ast,t)f(x,v,t)dv_\ast dvdydx\\
			&=-2(\kappa \psi_M)^2\int_{\mathbb{R}^{2d}}|x-x_c|^2f(x,v,t)dxdv-\frac{1}{2} \int_{\mathbb{R}^{2d}}|v-v_c|^2f(x,v,t)dxdv.
		\end{aligned} 
	\end{align} 
	Then, we combine  \eqref{1.6} and  \eqref{2.1} to get an estimate for the cross term:
	\begin{align}
		\begin{aligned}   \label{2.2}
			&\frac{d}{dt}\left( \int_{\mathbb{R}^{2d}}(x-x_c)\cdot(v-v_c)fdxdv\right)\\
			&\geq \frac{1}{2}\int_{\mathbb{R}^{2d}}|v-v_c|^2fdxdv -(1+2(\kappa \psi_M)^2)\int_{\mathbb{R}^{2d}}|x-x_c|^2fdxdv.
		\end{aligned} 
	\end{align}
	Finally, we take $\frac{1}{2}\eqref{l}+\frac{1}{2}\eqref{1.4}+\epsilon\eqref{2.2}$ to obtain a differential inequality $\tilde{\mathcal L}[f](t)$:
	\begin{align}
		\begin{aligned} \label{2.3}
			&\frac{d\tilde{\mathcal L}[f](t)}{dt}
			=\frac{d}{dt}\int_{\mathbb{R}^{2d}}\left(\frac{1}{2}|x-x_c|^2+\frac{1}{2}|v-v_c|^2+\epsilon (x-x_c)\cdot(v-v_c)\right)fdxdv\\
			&\geq-\left(\kappa\psi_M||f_0||_{L^1}-d\sigma-\frac{1}{2}\epsilon\right)\!\!\int_{\mathbb{R}^{2d}} |v-v_c|^2 fdxdv
			-(1+2(\kappa \psi_M)^2)\epsilon\int_{\mathbb{R}^{2d}} |x-x_c|^2 fdxdv.
		\end{aligned} 
	\end{align}
	One can simplify \eqref{2.3} by taking $\epsilon=\max \{-\frac{1}{2},\kappa\psi_M||f_0||_{L^1}-d\sigma\}$:
	\begin{align}
		\begin{aligned} \label{2.4}
			&\frac{d}{dt}\int_{\mathbb{R}^{2d}}\left(\frac{1}{2}|x-x_c|^2+\frac{1}{2}|v-v_c|^2+\epsilon (x-x_c) \cdot (v-v_c)\right)fdxdv\\
			& \geq C_M \int_{\mathbb{R}^{2d}} (|x-x_c|^2+|v-v_c|^2) fdxdv,
		\end{aligned} 
	\end{align}
	where $C_M=\min \{ -(1+2(\kappa \psi_M)^2)\epsilon,\frac{d\sigma-\kappa\psi_M||f_0||_{L^1}}{2}\}>0$.
	In \eqref{2.4}, we apply the Gr\"onwall's inequality to get
	\begin{align*}
		\begin{aligned} %\label{2.5}
			&\tilde{\mathcal L}[f](t) \geq \tilde{\mathcal L}[f_0]e^{\frac{4}{3}C_Mt}
		\end{aligned} 
	\end{align*}
	 and conclude the second result \eqref{nf}.
\end{proof}

\section{Mean-field limit: from stochastic particle system to the kinetic C-S Vlasov-type equation}\label{sec:4}
\setcounter{equation}{0}
In this section, we present a uniform-in-time mean-field limit from the stochastic  C-S model to the kinetic C-S Vlasov-type equation in a large population  limit $N\to\infty$ in the whole time-interval $[0,\infty)$, 
using  the so called C-S McKean process. Note that the Galilean invariance is hold because of harmonic oscillators \eqref{macro1} or \eqref{macro}.   Similar to  \cite{RE}, it is sufficient to study with $(x_c(0),v_c(0))=(0,0)$,  which implies  $(x_c(t),v_c(t))=(0,0)$.
% To  derive such a uniform-in-time mean-field limit estimate, we have to  perform a  very careful using the detailed information of the so called  Cuker-Smale- McKean process. We briefly discuss the concept.
Recall  that the stochastic particle model is
\begin{align}
	\left\{\begin{aligned}\label{SCS}
		&dx_t^i =v_t^i dt,   \qquad  t>0,  \quad  1 \leq i \leq N,  \\
		&dv_t^i = -\frac{\kappa}{N}\sum\limits_{j=1}^N \psi(|x_t^j-x_t^i|) (v_t^i-v_t^j) dt-x_t^idt+\sqrt{2\sigma} v_t^i dW_t,\\
		&\sum\limits_{i=1}^N x_t^i=0,\quad \sum\limits_{i=1}^N v_t^i=0. 
	\end{aligned}\right.
\end{align}
By symmetry of the initial configurations and the pairwise interaction of particles, all particles have the same distribution on $\mathbb{R}^{2d}$ at time $t$, which will be denoted  $f^{1,N}$.
 
 To study the behavior  of the  stochastic C-S model  \eqref{SCS}  for a large  population $N\gg1$, we use the concept of  ``propagation of chaos" originating from Kac's Markovian model of  gas dynamics.
 The propagation of chaos refers to the phenomenon that for any finite number of particles, each particles follow the  ``McKean process" $(\bar{x}_t^i,\bar{v}_t^i, f)$, which is given by the solution of the following  symmetric particle system
 \begin{align}
	\left\{\begin{aligned}\label{MK1}
		&d\bar{x}_t^i =\bar{v}_t^i dt,   \qquad  t>0,  \quad  1 \leq i \leq N,  \\
		&d\bar{v}_t^i = -\kappa  \psi*f(\bar{x}_t^i,\bar{v}_t^i) dt-\bar{x}_t^idt+\sqrt{2\sigma} \bar{v}_t^i dW_t,\\
		&f=law(\bar{x}_t^i,\bar{v}_t^i),\quad (\bar{x}_t^i(0),\bar{v}_t^i(0))=(x_0^i,v_0^i),
	\end{aligned}\right.
\end{align}
where $\psi*f(x,v):=\int_{\mathbb{R}^{2d}} \psi(|x-y|)(v-v_\ast) f(y,v_\ast)dydv_\ast$ is the convolution between $\psi$ and $f$ in the phase space.   By the way, when the  McKean process acts on the empirical measure $\mu^N_t=\frac{1}{N}\sum\limits_{i=1}^{N}\delta_{(x_t^i ,v_t^i)}$, the system  \eqref{MK1}  can be reduced  to the stochastic system \eqref{SCS}.

The system \eqref{MK1} consists of $N$ equations which can be solved  independently of each other. Each of them involves the condition that $f$ is the distribution of $(\bar{x}_t^i,\bar{v}_t^i)$, thus making it nonlinear.
The system  \eqref{MK1} and the system \eqref{SCS}  have the same initial configurations and share the same Brownian motions.  Note that the system  \eqref{MK1} is not anymore an SDE system, and the dynamics between particles is now coupled through the law $f$. The law is the same for each particle. It is straightforward to check that $f$ is just the (weak) solution to the mean-field PDE model  \eqref{km} by It\^{o}'s  formula. 
%So the process $(\bar{x}_t^i,\bar{v}_t^i)$ and the distribution $f$ should  be linked by the relation  $f=law(\bar{x}_t^i,\bar{v}_t^i)$.
According to \cite{BC, HKR},   in order to verify the weak convergence from the empirical measure $\mu^N_t$  to $f$, i.e.   the mean-field limit, it is sufficient  to show that for any $t$,
 $\lim\limits_{N\to\infty}(E[|x_t^i-\bar{x}_t^i|^2]+E[|v_t^i-\bar{v}_t^i|^2])=0$, under the assumption of the well-posedness of both SDE and PDE systems.
  %For a fixed $\mu\in\mathcal{P}(\mathbf{R}^{2d})$, the empirical measure associated with $f$ 
   \begin{remark}
      The estimate $\lim\limits_{N\to\infty}(E[|x_t^i-\bar{x}_t^i|^2]+E[|v_t^i-\bar{v}_t^i|^2])=0$ classically  ensures quantitative estimates on ( see \cite{BC, HKR} for details)
         \begin{enumerate}
        \item 
      The convergence in $N$ of the law  $f^{1,N}$ at time $t$ of any (by symmetry) of the processes $(x_t^i ,v_t^i)$ towards $f$;
        \item 
         The propagation of chaos: for all fixed $k$, the law  $f^{k,N}$  for any $k$ particles $(x_t^i ,v_t^i)$  converges to the tensor  product $f^{\otimes k}$ as $N$ tends to infinity;
       \item 
         The convergence of the empirical measure $\mu^N_t$ at time $t$ of the particle system \eqref{SCS}  towards $f$.
         \end{enumerate}
    \end{remark}
    \subsection{Exponential decay in time estimate for the  limit process}
    In this subsection, we derive a uniform-in-time boundedness  and decay property for the difference between  two processes \eqref{SCS} and \eqref{MK1} driven by the same  Brownian motion and initial configuration. 
  
  Firstly, we need  to derive the  flocking  and concentration result in the probabilistic sense  for the  McKean process $(\bar{x}_t^i,\bar{v}_t^i)$ for  preparation. 
We set 
\begin{align*}
	\begin{aligned}\label{E}
		& a(x,t):= \int_{\mathbb{R}^{2d}} \psi(|x-y|)f(y,v_\ast,t)dv_\ast dy,\\
		& b(x,t):= \int_{\mathbb{R}^{2d}} v_\ast \psi(|x-y|)f(y,v_\ast,t)dv_\ast dy.
	\end{aligned} 
\end{align*}
Then we have 
\begin{align}
	\begin{aligned}
		&  \psi_m ||f_0||_{L^1} \leq a(x,t) \leq \psi_M ||f_0||_{L^1},\\
		&  |b(x,t)| \leq 2\psi_M e^{-\frac{2C_mt}{3}}\sqrt{||f_0||_{L^1}||(|v|^2+|x|^2)f_0||_{L^1}},\\
	\end{aligned} 
\end{align} 
and \eqref{MK1} becomes
\begin{align}
	\left\{\begin{aligned}\label{MK2}
		&d\bar{x}_t^i =\bar{v}_t^i dt,   \qquad  t>0,  \quad  1 \leq i \leq N,  \\
		&d\bar{v}_t^i = -\kappa (a(\bar{x}_t^i,t)\bar{v}_t^i-b(\bar{x}_t^i,t)) dt-\bar{x}_t^idt+\sqrt{2\sigma} \bar{v}_t^i dW_t,\\
		&f=law(\bar{x}_t^i,\bar{v}_t^i),\quad (\bar{x}_t^i(0),\bar{v}_t^i(0))=(x_0^i,v_0^i).
	\end{aligned}\right.
\end{align}
	 Now we derive a  differential inequality.
\begin{lemma}
	Suppose that the communication weight function $\psi$, $\kappa$, $\sigma$ and initial configuration $f_0$ satisfy the conditions: there exist positive constants $\psi_m,\psi_M$ such that
	\begin{align*}
		\begin{aligned} 
			& 0< \psi_m \leq \psi \leq \psi_M,  \quad \kappa\psi_m||f_0||_{L^1} > d\sigma,\quad	 \int_{\mathbb{R}^{2d}}(1+|v|^2+|x|^2)f_0dxdv<\infty.
		\end{aligned}
	\end{align*}
	Then, we have 
	\begin{align*}
		\begin{aligned} 
				 \frac{d}{dt}E\left[\frac{1}{2}|\bar{x}_t^i|^2+\frac{1}{2}|\bar{v}_t^i|^2+\epsilon\bar{x}_t^i\cdot\bar{v}_t^i\right]
				\leq -\eta E\left[|\bar{x}_t^i|^2+|\bar{v}_t^i|^2\right]+\lambda e^{-\frac{4}{3}C_m t},
		\end{aligned}
	\end{align*}
where  $\epsilon=\min \{\frac{1}{2},\frac{\kappa\psi_m||f_0||_{L^1}-\sigma}{2(1+2(\kappa\psi_M)^2)}\}$, $\eta = \min\{\frac{1}{2}\epsilon,\frac{1}{4}(\kappa \psi_m ||f_0||_{L^1}-\sigma) \}$, $\delta=\frac{\kappa\psi_m||f_0||_{L^1}-\sigma}{\kappa}$, $\lambda = (\frac{2}{\delta}+4\epsilon\kappa)\kappa\psi_M^2 ||f_0||_{L^1}||(|v|^2+|x|^2)f_0||_{L^1}$ and $C_m$ is given in Theorem \ref{u4}.
\end{lemma}
\begin{proof}

\noindent $\bullet$ (Estimate for $ d|\bar{x}_t^i|^2$):~ By It\^{o}'s formula, we can obtain 
\begin{align}
	\begin{aligned} \label{q1}
		&d|\bar{x}_t^i|^2 = 2 \bar{x}_t^i\cdot d\bar{x}_t^i + d\bar{x}_t^i \cdot d\bar{x}_t^i=  2\bar{x}_t^i\cdot  \bar{v}_t^idt.
	\end{aligned}
\end{align}

\noindent $\bullet$ (Estimate for $ d|\bar{v}_t^i|^2$):~By direct calculation, we have 
\begin{align*}
	\begin{aligned} 
		d|\bar{v}_t^i|^2= 2 \bar{v}_t^i \cdot d\bar{v}_t^i+d\bar{v}_t^i \cdot d\bar{v}_t^i 
		:= \mathcal{I}_{11}+\mathcal{I}_{12}.\\
	\end{aligned}
\end{align*}
We use \eqref{MK2}$_2$ and \eqref{E} to obtain 
\begin{align}
	\begin{aligned}\label{aa}
		\mathcal{I}_{11}&=2 \bar{v}_t^i\cdot d\bar{v}_t^i\\
		& = -2\kappa(a(\bar{x}_t^i,t)|\bar{v}_t^i|^2-b(\bar{x}_t^i,t)\bar{v}_t^i) dt-2 \bar{x}_t^i\cdot \bar{v}_t^idt+2\sqrt{2\sigma}|\bar{v}_t^i|^2dW_t\\
		& \leq -2\kappa\left(\psi_m||f_0||_{L^1}-\frac{\delta}{2}\right)|\bar{v}_t^i|^2dt+\frac{4\kappa\psi_M^2}{\delta}e^{-\frac{4}{3}C_mt}||f_0||_{L^1}||(|v|^2+|x|^2)f_0||_{L^1}dt\\
		&\quad -2 \bar{x}_t^i\cdot \bar{v}_t^idt+2\sqrt{2\sigma}|\bar{v}_t^i|^2 dW_t,
	\end{aligned}
\end{align}
where we used the $\delta$-Young inequailty and $\delta$ will be defined later:
\begin{align*}
	\begin{aligned}
	   |\bar{v}_t^i||b(\bar{x}_t^i,t)|  \leq \frac{\delta}{2}|\bar{v}_t^i|^2+ \frac{2}{\delta}\psi_M^2e^{-\frac{4}{3}C_mt}||f_0||_{L^1}||(|v|^2+|x|^2)f_0||_{L^1}.
	\end{aligned}
\end{align*}
%Then we use
%\begin{align}
%	\begin{aligned} \label{}
%		& \frac{1}{N}\sum_{i=1}^N \bar{v}_t^i^2= \frac{(\sum_{i=1}^N v_t^i)^2}{N^2}-\frac{1}{N}\sum_{i=1}^N|\bar{v}_t^i|^2 \leq 0,
	%\end{aligned}
%\end{align}
%and finally we obtain
%\begin{align}
	%\begin{aligned}\label{aa}
		%\mathcal{I}_{11}
	%	%& \leq -2\kappa(\psi_m||f_0||_{L^1}\mathcal V_t-\psi_M \sqrt{\mathcal V_t} e^{-\frac{C_mt}{2}}\sqrt{||f_0||_{L^1}||(|v|^2+|x|^2)f_0||_{L^1}})dt -\frac{2}{N}\sum_{i=1}^N \bar{x}_t^i\bar{v}_t^idt.
	%\end{aligned}
%\end{align}

The term $\mathcal{I}_{12}$ can be easily obtained as follows:
\begin{align}
	\begin{aligned} \label{bb}
		\mathcal{I}_{12}&=d\bar{v}_t^i\cdot d\bar{v}_t^i
		= 2\sigma|\bar{v}_t^i|^2dt.
		%& = 2\sigma[\frac{1}{N}\sum_{i=1}^N|\bar{v}_t^i|^2-\frac{1}{N^2}(\sum_{i=1}^N v_t^i)^2]dt\\
	\end{aligned}
\end{align}
 We combine \eqref{aa} and \eqref{bb}  to have
\begin{align}
	\begin{aligned}\label{q2}
		d|\bar{v}_t^i|^2 & \leq \!-2\kappa\!\left(\psi_m||f_0||_{L^1}\!-\!\frac{\delta}{2}\!-\!\frac{\sigma}{\kappa}\!\right)|\bar{v}_t^i|^2dt\!+\frac{4\kappa\psi_M^2}{\delta}e^{-\frac{4}{3}C_mt}||f_0||_{L^1}||(|v|^2+|x|^2)f_0||_{L^1}dt\\
		&\quad -2\bar{x}_t^i\cdot \bar{v}_t^idt+2\sqrt{2\sigma}|\bar{v}_t^i|^2dW_t.
	\end{aligned}
\end{align}
\noindent $\bullet$ (Estimate for $ d (\bar{x}_t^i\cdot \bar{v}_t^i)$):~Now we estimate the cross term:
\begin{align}
	\begin{aligned} \label{53}
		d(\bar{x}_t^i\cdot \bar{v}_t^i)&= \bar{x}_t^i\cdot d\bar{v}_t^i +\bar{v}_t^i\cdot d\bar{x}_t^i+d\bar{x}_t^i\cdot d\bar{v}_t^i \\
		&= \bar{x}_t^i\cdot [-\kappa (a(\bar{x}_t^i,t)\bar{v}_t^i-b(\bar{x}_t^i,t)) dt-\bar{x}_t^idt+\sqrt{2\sigma} \bar{v}_t^i dW_t]+|\bar{v}_t^i|^2dt\\
		&= -\kappa\bar{x}_t^i\cdot (a(\bar{x}_t^i,t)\bar{v}_t^i-b(\bar{x}_t^i,t))dt+\sqrt{2\sigma} \bar{x}_t^i\cdot \bar{v}_t^idW_t-|\bar{x}_t^i|^2dt+|\bar{v}_t^i|^2 dt\\
		&:= \mathcal{I}_{21}dt+\sqrt{2\sigma} \bar{x}_t^i\cdot \bar{v}_t^idW_t-|\bar{x}_t^i|^2dt+|\bar{v}_t^i|^2 dt.
	\end{aligned}
\end{align}
$\mathcal{I}_{21}$ can be estimated as follows:
\begin{align}
	\begin{aligned} \label{51}
		& |\mathcal{I}_{21}| = \kappa|\bar{x}_t^i\cdot (a(\bar{x}_t^i,t)\bar{v}_t^i-b(\bar{x}_t^i,t))| \\
		& \leq \kappa \psi_M ||f_0||_{L^1}\left(\kappa \psi_M ||f_0||_{L^1}|v_t^i|^2+\frac{1}{4\kappa \psi_M ||f_0||_{L^1}}|x_t^i|^2\right)
		+\kappa|\bar{x}_t^i| |b(\bar{x}_t^i,t)|\\
		& \leq  \left[\frac{1}{4}|\bar{x}_t^i|^2+(\kappa \psi_M ||f_0||_{L^1})^2|\bar{v}_t^i|^2\right]\!+\!2\kappa|\bar{x}_t^i| \psi_M e^{-\frac{2C_mt}{3}}\!\sqrt{||f_0||_{L^1}||(|v|^2+|x|^2)f_0||_{L^1}}\\
		&\leq  \left[\frac{1}{2}|\bar{x}_t^i|^2+(\kappa \psi_M ||f_0||_{L^1})^2|\bar{v}_t^i|^2\right]+4\kappa^2\psi_M^2e^{-\frac{4}{3}C_mt} ||f_0||_{L^1}||(|v|^2+|x|^2)f_0||_{L^1}.
	\end{aligned}
\end{align}
By inserting \eqref{51}  into \eqref{53}, we have 
\begin{align}
	\begin{aligned}\label{q3}
		d(\bar{x}_t^i\cdot \bar{v}_t^i)&\leq -\frac{1}{2}|\bar{x}_t^i|^2dt+ \left[(\kappa \psi_M ||f_0||_{L^1})^2+1\right]|\bar{v}_t^i|^2 dt\\
		& \quad+4\kappa^2\psi_M^2e^{-\frac{4}{3}C_mt} ||f_0||_{L^1}||(|v|^2+|x|^2)f_0||_{L^1}dt+\sqrt{2\sigma} \bar{x}_t^i\cdot \bar{v}_t^idW_t.
		\end{aligned}
\end{align}

Now we take a combination  $\frac{1}{2}\eqref{q1}+\frac{1}{2}\eqref{q2}+\epsilon\eqref{q3}$ to get the following differential inequality:
\begin{align*}
	\begin{aligned}% \label{}
		& dE\left(\frac{1}{2}|\bar{x}_t^i|^2+\frac{1}{2}|\bar{v}_t^i|^2+\epsilon \bar{x}_t^i\cdot \bar{v}_t^i\right)\\
		&\leq -\left[\kappa \psi_m ||f_0||_{L^1}-\sigma-\frac{\delta\kappa}{2}-((\kappa\psi_M||f_0||_{L^1})^2+1)\epsilon\right]E[|\bar{v}_t^i|^2]dt - \frac{\epsilon}{2}E[|\bar{x}_t^i|^2] dt\\
		& \quad +\left(\frac{2}{\delta}+4\epsilon\kappa\right)\kappa\psi_M^2e^{-\frac{4}{3}C_mt} ||f_0||_{L^1}||(|v|^2+|x|^2)f_0||_{L^1}dt.
	\end{aligned}
\end{align*}
We take $\delta:=\frac{\kappa\psi_m||f_0||_{L^1}-\sigma}{\kappa}$, $\epsilon: = \min\{\frac{1}{2},\frac{1}{4}\frac{\kappa \psi_m ||f_0||_{L^1}-\sigma}{(\kappa \psi_M ||f_0||_{L^1})^2+1}\}$ and obtain
\begin{align}
	\begin{aligned} \label{f}
		dE\left[\frac{1}{2}|\bar{x}_t^i|^2+\frac{1}{2}|\bar{v}_t^i|^2+\epsilon\bar{x}_t^i\cdot \bar{v}_t^i\right]
	 \leq -\eta E\left[|\bar{x}_t^i|^2+|\bar{v}_t^i|^2\right]dt+\lambda e^{-\frac{4}{3}C_m t} dt,\\
	\end{aligned}
\end{align}
where $\eta := \min\{\frac{1}{2}\epsilon,\frac{1}{4}(\kappa \psi_m ||f_0||_{L^1}-\sigma) \}, \lambda := (\frac{2}{\delta}+4\epsilon\kappa)\kappa\psi_M^2 ||f_0||_{L^1}||(|v|^2+|x|^2)f_0||_{L^1}$.
\end{proof}
\begin{lemma}\label{i5}
Suppose that the communication weight function $\psi$, $\kappa$, $\sigma$ and initial configuration $f_0$ satisfy the conditions: there exist positive constants $\psi_m,\psi_M$ such that
 \begin{align*}
 	\begin{aligned} 
 		& 0 < \psi_m \leq \psi \leq \psi_M, \quad \kappa\psi_m||f_0||_{L^1} >d\sigma, \quad \int_{\mathbb{R}^{2d}}(1+|v|^2+|x|^2)f_0dxdv<\infty, \quad t\geq0.
 	\end{aligned}
 \end{align*}
Then we have
	\begin{align*}
		\begin{aligned}
			& E[|\bar{x}_t^i|^2+|\bar{v}_t^i|^2] \leq Ce^{-\frac{4}{3}C_\ast t},
		\end{aligned}
	\end{align*}
where $C_\ast=\min\{C_m,\eta\}$ and $C$ is a constant which depends on initial configuration, $\psi$, $\kappa$ and $\sigma$.
\end{lemma}
\begin{proof}
We set 
	\begin{align*}
		\begin{aligned} %\label{}
			& \mathcal Z_t := |\bar{x}_t^i|^2+|\bar{v}_t^i|^2.
		\end{aligned}
	\end{align*}
Since $|\epsilon| \leq \frac{1}{2}$, the quantity $\frac{1}{2}|\bar{x}_t^i|^2+\frac{1}{2}|\bar{v}_t^i|^2+\epsilon \bar{x}_t^i\cdot \bar{v}_t^i$ is equivalent to $|\bar{x}_t^i|^2+|\bar{v}_t^i|^2$ :
\[  \frac{3}{16}\mathcal Z_t \le  \frac{1}{2}|\bar{x}_t^i|^2+\frac{1}{2}|\bar{v}_t^i|^2+\epsilon\bar{x}_t^i\cdot \bar{v}_t^i\le  \frac{3}{4}\mathcal Z_t .\]
Then, it follows from \eqref{f} that  $\frac{1}{2}|\bar{x}_t^i|^2+\frac{1}{2}|\bar{v}_t^i|^2+\epsilon \bar{x}_t^i\cdot \bar{v}_t^i$ satisfies the following SDE:	
	\begin{align*}
		\begin{aligned}
	dE\left[\frac{1}{2}|\bar{x}_t^i|^2+\frac{1}{2}|\bar{v}_t^i|^2+\epsilon \bar{x}_t^i\cdot \bar{v}_t^i\right]\leq -\frac{4}{3}\eta E\left[\frac{1}{2}|\bar{x}_t^i|^2+\frac{1}{2}|\bar{v}_t^i|^2+\epsilon \bar{x}_t^i\cdot \bar{v}_t^i\right]dt +\lambda e^{-\frac{4}{3}C_m t}dt.
		\end{aligned}
	\end{align*}
	Therefore, we obtain
\begin{align*}
	\begin{aligned} 
		&E[\mathcal Z_t] \leq 4E[\mathcal Z_0] e^{-\frac{4}{3}\eta t}+\frac{4\lambda }{\eta-C_m}(e^{-\frac{4}{3}C_m t}-e^{-\frac{4}{3}\eta t})\le Ce^{-\frac{4}{3}C_\ast t},
	\end{aligned}
\end{align*}
where $C_\ast=\min\{C_m,\eta\}$ and $C$ is a general constant. Then we get the desired result.
\end{proof}
 
 Now we are ready to state our result on the  exponential decay in time estimate for the limit process.
\begin{theorem}\label{4.1}
Suppose that the communication weight function $\psi$, $\kappa$, $\sigma$ and initial configuration $f_0$ satisfy the conditions: there exist positive constants $\psi_m,\psi_M$ such that
\begin{align*}
	\begin{aligned} %\label{}
		& 0 < \psi_m \leq \psi \leq \psi_M, \quad \kappa\psi_m\min\{||f_0||_{L^1},1\}>d\sigma, \quad  \int_{\mathbb{R}^{2d}}(1+|v|^2+|x|^2)f_0dxdv<\infty
	\end{aligned}
\end{align*}
and let $(x_t^i,v_t^i)$ and $(\bar{x}_t^i,\bar{v}_t^i,f)$ be solution processes to the systems in \eqref{SCS} and \eqref{MK2}, respectively.
Then, we have 
	\begin{align*}
		\begin{aligned} %\label{}
			& E[|x_t^i-\bar{x}_t^i|^2]+E[|v_t^i-\bar{v}_t^i|^2] \leq Ce^{-C_3 t},
		\end{aligned}
	\end{align*}
where  $C$ and $C_3$ depend on initial configuration, $\psi$, $\kappa$ and $\sigma$.
	% for a general constant $C$ depends on $f_0$, $\psi$, $\kappa$ and $\sigma$. Here $C_m$ is given in Theorem \ref{u4}.
\end{theorem}

To prove this theorem, we first introduce the following functional:
	\begin{align*}
		\begin{aligned}%\label{s}
			& \tilde{\mathfrak{L}}(t) = E\left[\frac{1}{2}|x_t^i-\bar{x}_t^i|^2+\frac{1}{2}|v_t^i-\bar{v}_t^i|^2+\epsilon(x_t^i-\bar{x}_t^i)\cdot (v_t^i-\bar{v}_t^i)\right], \quad |\epsilon|\leq \frac{1}{2},
		\end{aligned}
	\end{align*}
which is equivalent to $\mathfrak L(t)=E[|x_t^i-\bar{x}_t^i|^2+|v_t^i-\bar{v}_t^i|^2]$.\\
Now we  derive a differential equation for $\tilde{\mathfrak L}(t)$.

	\noindent $\bullet$ (Estimate for $dE[|x_t^i-\bar{x}_t^i|^2]$): By straightforward calculation, we have
	\begin{align}
		\begin{aligned} \label{x}
			%&d\sum_{i=1}^N |x_t^i-\bar{x}_t^i|^2 = 2\sum_{i=1}^N (x_t^i-\bar{x}_t^i)d(x_t^i-\bar{x}_t^i) +\sum_{i=1}^N d(x_t^i-\bar{x}_t^i)d(x_t^i-\bar{x}_t^i).
			&dE[|x_t^i-\bar{x}_t^i|^2] = 2 E[(x_t^i-\bar{x}_t^i)\cdot(v_t^i-\bar{v}_t^i)]dt.
		\end{aligned}
	\end{align}

	\noindent $\bullet$ (Estimate for $dE[|v_t^i-\bar{v}_t^i|^2]$): It follows from \eqref{SCS} and \eqref{MK2} that $v_t^i-\bar{v}_t^i$ satisfies
	\begin{align*}
		\begin{aligned} 
			&d(v_t^i-\bar{v}_t^i) \\
			&= -\frac{\kappa}{N}\sum_{j=1}^N[\psi(|x_t^j-x_t^i|) (v_t^i-v_t^j)-a(\bar{x}_t^i,t)\bar{v}_t^i+b(\bar{x}_t^i,t)]dt\\
			&\quad-(x_t^i-\bar{x}_t^i)dt+\sqrt{2\sigma}(v_t^i-\bar{v}_t^i) dW_t.\\
		\end{aligned}
	\end{align*}
By It\^{o}'s formula, we can obtain
	\begin{align*}
		\begin{aligned} %\label{}
			&d|v_t^i-\bar{v}_t^i|^2 = 2(v_t^i-\bar{v}_t^i)\cdot d(v_t^i-\bar{v}_t^i) +d(v_t^i-\bar{v}_t^i)\cdot d(v_t^i-\bar{v}_t^i)\\
			&=-\frac{2\kappa}{N}\sum_{j=1}^N(v_t^i-\bar{v}_t^i)\cdot [\psi(|x_t^j-x_t^i|) (v_t^i-v_t^j)-a(\bar{x}_t^i,t)\bar{v}_t^i+b(\bar{x}_t^i,t)]dt-2(x_t^i-\bar{x}_t^i)\cdot(v_t^i-\bar{v}_t^i)dt\\
			&\quad+2\sqrt{2\sigma}|v_t^i-\bar{v}_t^i|^2dW_t+2\sigma|v_t^i-\bar{v}_t^i|^2dt.
		\end{aligned}
	\end{align*}
We use 
	\begin{align*}
		\begin{aligned} 
		  &E[|v_t^i-\bar{v}_t^i|^2dW_t]=0
		\end{aligned}
	\end{align*}
 to obtain
	\begin{align}
		\begin{aligned} \label{hh}
			&\frac{d}{dt} E[|v_t^i-\bar{v}_t^i|^2] =-\frac{2\kappa}{N}E\left[\sum_{j=1}^N(v_t^i-\bar{v}_t^i)\cdot(\psi(|x_t^j-x_t^i|)(v_t^i-v_t^j)-a(\bar{x}_t^i,t)\bar{v}_t^i+b(\bar{x}_t^i,t))\right]\\
			&\quad +2\sigma E[|v_t^i-\bar{v}_t^i|^2]-2E[(x_t^i-\bar{x}_t^i)\cdot(v_t^i-\bar{v}_t^i)]\\
			&:= \mathcal{I}_3+2\sigma E[|v_t^i-\bar{v}_t^i|^2]-2E[(x_t^i-\bar{x}_t^i)\cdot(v_t^i-\bar{v}_t^i)].
		\end{aligned}
	\end{align}
We further decompose the term $\mathcal{I}_3$ as follows:
	\begin{align*}
		\begin{aligned} 
			&\mathcal{I}_3 = -\frac{2\kappa}{N}E\left[\sum_{j=1}^N(v_t^i-\bar{v}_t^i)\cdot\psi(|x_t^j-x_t^i|) \left((v_t^i-v_t^j)-(\bar{v}_t^i-\bar{v}_t^j)\right)\right]+\frac{2\kappa}{N} E[(v_t^i-\bar{v}_t^i)\cdot a(\bar{x}_t^i,t)\bar{v}_t^i ]\\
			& \qquad -\frac{2\kappa}{N}E\left[\sum_{j\neq i}^N(v_t^i-\bar{v}_t^i)\cdot\left(\psi(|x_t^j-x_t^i|) (\bar{v}_t^i-\bar{v}_t^j)-a(\bar{x}_t^i,t)\bar{v}_t^i\right)\right]-\frac{2\kappa}{N}\sum_{j=1}^N E[(v_t^i-\bar{v}_t^i)\cdot b(\bar{x}_t^i,t)]\\
			& \quad:= \mathcal{I}_{31}+\mathcal{I}_{32}+\mathcal{I}_{33}+\mathcal{I}_{34}.
		\end{aligned}
	\end{align*}
\begin{lemma}
	The terms $\mathcal{I}_{3i}$, $i=1,\cdots,4$, satisfy the following estimates:
	\begin{align*}
		\begin{aligned} 
			& \mathcal{I}_{31}(t) \leq -2\kappa\psi_m E[|v_t^i-\bar{v}_t^i|^2]+2\kappa\psi_m\sqrt{E[|\bar{v}_t^i|^2]}\sqrt{E[|v_t^i-\bar{v}_t^i|^2]},\\
			& \mathcal{I}_{32}(t) \leq\frac{2\kappa \psi_M ||f_0||_{L^1}}{N}\sqrt{E[|v_t^i-\bar{v}_t^i|^2]}\sqrt{E[|\bar{v}_t^i|^2]},\\
			& \mathcal{I}_{33}(t) \leq  4\kappa \psi_M (1+||f_0||_{L^1})\sqrt{\frac{N-1}{N}} \sqrt{E[|v_t^i-\bar{v}_t^i|^2]}\sqrt{ E[|\bar{v}_t^i|^2]},\\
			& \mathcal{I}_{34}(t) \leq 4\kappa \psi_M e^{-\frac{2C_mt}{3}}\sqrt{||f_0||_{L^1} ||(|v|^2+|x|^2)f_0||_{L^1}}\sqrt{E[|v_t^i-\bar{v}_t^i|^2]}.
		\end{aligned}
	\end{align*}
\end{lemma}
\begin{proof} The proof of this lemma is similar to that in \cite{HKR}.
	We first use the symmetry to estimate $\mathcal{I}_{31}$ as follows:
	\begin{align*}
		\begin{aligned}
			&\mathcal{I}_{31} = -\frac{2\kappa}{N^2}\sum_{i,j=1}^N E\left[(v_t^i-\bar{v}_t^i)\cdot\psi(|x_t^j-x_t^i|) \left((v_t^i-v_t^j)-(\bar{v}_t^i-\bar{v}_t^j)\right)\right].\\
		\end{aligned}
	\end{align*}
	We again use the trick $i \leftrightarrow j$ to obtain
	\begin{align*}
		\begin{aligned}
			 \mathcal{I}_{31} &=- \frac{2\kappa}{N^2}\sum_{i,j=1}^N E\left[(v_t^j-\bar{v}_t^j)\cdot\psi(|x_t^j-x_t^i|) \left((v_t^j-v_t^i)-(\bar{v}_t^j-\bar{v}_t^i)\right)\right]\\
			&= -\frac{\kappa}{N^2}\sum_{i,j=1}^N E\left[\psi(|x_t^j-x_t^i|)|(v_t^i-\bar{v}_t^i)-(v_t^j-\bar{v}_t^j)|^2\right]\\
			&  \leq -2\kappa\psi_m E\left[|v_t^i-\bar{v}_t^i|^2\right]+\frac{2\kappa\psi_m}{N^2}\sum_{i,j=1}^N E\left[(v_t^i-\bar{v}_t^i)\cdot(v_t^j-\bar{v}_t^j)\right]\\
			& : = -2\kappa\psi_m E\left[|v_t^i-\bar{v}_t^i|^2\right]+\mathcal{F}.\\
		\end{aligned}
	\end{align*}
We estimate $\mathcal{F}$ and get 
	\begin{align*}
		\begin{aligned} 
			 \mathcal{F} &= \frac{2\kappa\psi_m}{N^2}\sum_{i,j=1}^N E\left[(v_t^i-\bar{v}_t^i)\cdot(v_t^j-\bar{v}_t^j)\right]\\
			&= -\frac{2\kappa\psi_m}{N^2}\sum_{i,j=1}^NE\left[ (v_t^i-\bar{v}_t^i)\cdot\bar{v}_t^j\right]\\
			%& \quad \leq \frac{2\kappa\psi_m}{N} \sqrt{E[|\sum_{i=1}^N (v_t^i-\bar{v}_t^i)|^2]}\sqrt{E[|\sum_{i=1}^N \bar{v}_t^i|^2]}\\
			%& \quad \leq \frac{2\kappa\psi_m}{N} \sqrt{NE[\sum_{i=1}^N|v_t^i-\bar{v}_t^i|^2]}\sqrt{NE[\sum_{i=1}^N|\bar{v}_t^i|^2]}\\
			&\leq \frac{2\kappa\psi_m}{N^2}\sum_{i,j=1}^N\sqrt{E\left[|\bar{v}_t^j|^2\right]}\sqrt{E\left[|v_t^i-\bar{v}_t^i|^2\right]}\\
			& \leq 2\kappa\psi_m\sqrt{E\left[|\bar{v}_t^j|^2\right]}\sqrt{E\left[|v_t^i-\bar{v}_t^i|^2\right]},
		\end{aligned}
	\end{align*}
	where we used $\sum\limits_{j=1}^N v_t^j = 0$ and  $\sum\limits_{i=1}^Na_i\le\sqrt{N\sum\limits_{i=1}^Na_i^2}$.
	Therefore we get
	\begin{align*}
		\begin{aligned}
			&\mathcal{I}_{31} \leq -2\kappa\psi_m E\left[|v_t^i-\bar{v}_t^i|^2\right]+2\kappa\psi_m\sqrt{E[|\bar{v}_t^j|^2]}\sqrt{E[|v_t^i-\bar{v}_t^i|^2]}.
		\end{aligned}
	\end{align*}
\quad 	Now we estimate $\mathcal{I}_{32}$:
	\begin{align*}
		\begin{aligned}
			\mathcal{I}_{32}&= \frac{2\kappa}{N} E\left[(v_t^i-\bar{v}_t^i) a(\bar{x}_t^i,t)\bar{v}_t^i \right] \leq \frac{2\kappa \psi_M ||f_0||_{L^1}}{N}\sqrt{E\left[|v_t^i-\bar{v}_t^i|^2\right]}\sqrt{E\left[|\bar{v}_t^i|^2\right]}.
		%	&=\frac{2\kappa \psi_M ||f_0||_{L^1}}{N}\sqrt{E[|v_t^i-\bar{v}_t^i|^2]}\sqrt{E[|\bar{v}_t^i|^2]}.
			%& \quad \leq \frac{2\kappa \psi_M ||f_0||_{L^1}}{\sqrt{N}}\sqrt{E[|\sum_{i=1}^N(v_t^i-\bar{v}_t^i)|^2]}\sqrt{E[|\bar{v}_t^i|^2]}\\
			%& \quad \leq \frac{2\kappa \psi_M ||f_0||_{L^1}}{\sqrt{N}}\sqrt{E[\sum_{i=1}^N|v_t^i-\bar{v}_t^i|^2]}\sqrt{E[|\bar{v}_t^i|^2]}.
		\end{aligned}
	\end{align*}

	Next we estimate $\mathcal{I}_{33}$ as follows:	
	\begin{align*}
		\begin{aligned} 
			&\mathcal{I}_{33} = -\frac{2\kappa}{N}E\left[\sum_{j\neq i}^N(v_t^i-\bar{v}_t^i)\cdot\left(\psi(|x_t^i-x_t^j|) (\bar{v}_t^i-\bar{v}_t^j)-a(\bar{x}_t^i,t)\bar{v}_t^i\right)\right].
		\end{aligned}
	\end{align*}
By symmetry and without loss of generality, we may assume $i=1$. We set
	\begin{align*}
		\begin{aligned} 
			 \Psi^j &:= \psi(|x_t^1-x_t^j|) (\bar{v}_t^1-\bar{v}_t^j)-a(\bar{x}_t^1,t)\bar{v}_t^1\\
			& = \left(\psi(|x_t^1-x_t^j|)-a(\bar{x}_t^1,t)\right)\bar{v}_t^1-\psi(|x_t^1-x_t^j|)\bar{v}_t^j, \quad j=2,\cdots,N.
		\end{aligned}
	\end{align*}
Then we have
	\begin{align*}
		\begin{aligned} 
			&\sum_{j=2}^N |\Psi^j|^2 \leq \sum_{j=2}^N \left(|\psi(|x_t^1-x_t^j|)-a(\bar{x}_t^1,t)||\bar{v}_t^1|+\psi(|x_t^1-x_t^j|)|\bar{v}_t^j|\right)^2\\
			& \leq \sum_{j=2}^N \max \left\{| \psi(|x_t^1-x_t^j|)-a(\bar{x}_t^1,t)|,\psi(|x_t^1-x_t^j|)\right\}^2\left(|\bar{v}_t^1|+|\bar{v}_t^j|\right)^2\\
			&\leq 2 \max \left\{| \psi(|x_t^1-x_t^j|)-a(\bar{x}_t^1,t)|,\psi(|x_t^1-x_t^j|)\right\}^2\left(\sum_{i=1}^N |\bar{v}_t^i|^2+(N-2)|\bar{v}_t^1|^2\right).
		\end{aligned}
	\end{align*}
	This yields 
	\begin{align}
		\begin{aligned} \label{uu}
			& \sum_{j=2}^N E\left[|\Psi^j|^2\right] \leq 4N\psi_M^2\left(1+||f_0||_{L^1}\right)^2E\left[|\bar{v}_t^i|^2\right].\\
		\end{aligned}
	\end{align}
	We now use \eqref{uu} to obtain
	\begin{align*}
		\begin{aligned} 
			 \mathcal{I}_{33} &= -\frac{2\kappa}{N}E\left[\sum_{j= 2}^N(v_t^i-\bar{v}_t^i)\cdot\Psi^j\right]\leq \frac{2\kappa}{N} \sqrt{E\left[|v_t^i-\bar{v}_t^i|^2\right]}\sqrt{E\left[\left|\sum_{j= 2}^N \Psi^j\right|^2\right]}\\
			& \leq \frac{2\kappa \sqrt{N-1}}{N}\sqrt{E\left[|v_t^i-\bar{v}_t^i|^2\right]}\sqrt{\sum_{j= 2}^N E\left[|\Psi^j|^2\right]}\\
			& \leq \frac{4\kappa \psi_M (1+||f_0||_{L^1})\sqrt{N} \sqrt{N-1}}{N}  \sqrt{E\left[|v_t^i-\bar{v}_t^i|^2\right]}\sqrt{ E\left[|\bar{v}_t^i|^2\right]}.
		%	& \quad \leq \frac{4\kappa \psi_M (1+||f_0||_{L^1})\sqrt{N} \sqrt{N-1}}{N}  \sqrt{NE[\sum_{i= 1}^N|v_t^i-\bar{v}_t^i|^2]}\sqrt{ E[|\bar{v}_t^i|^2]}\\
			%&\leq 4\kappa \psi_M (1+||f_0||_{L^1})\sqrt{N-1}  \sqrt{E[\sum_{i= 1}^N|v_t^i-\bar{v}_t^i|^2]}\sqrt{ E[|\bar{v}_t^i|^2]}\\
		\end{aligned}
	\end{align*}

	Finally we estimate $\mathcal{I}_{34}$ as follows:	
	\begin{align*}
		\begin{aligned}% \label{}
			|\mathcal{I}_{34}|&=\left|-\frac{2\kappa}{N}\sum_{j=1}^N E\left[(v_t^i-\bar{v}_t^i)\cdot b(\bar{x}_t^i,t)\right]\right|\le 2\kappa E\left[ |v_t^i-\bar{v}_t^i||b(\bar{x}_t^i,t)|\right]\\
			&\leq 4\kappa\psi_M e^{-\frac{2C_mt}{3}}\sqrt{||f_0||_{L^1} ||(|v|^2+|x|^2)f_0||_{L^1}}\sqrt{E\left[| v_t^i-\bar{v}_t^i|^2\right]}.
			%&\leq 2\kappa \sqrt{N} \sqrt{E[|b(\bar{x}_t^i,t)|^2]}\sqrt{E[\sum_{i=1}^N| v_t^i-\bar{v}_t^i|^2]}\\
			%&\leq 4\kappa \psi_M e^{-\frac{C_mt}{2}}\sqrt{N||f_0||_{L^1} ||(|v|^2+|x|^2)f_0||_{L^1}}\sqrt{E[\sum_{i=1}^N|v_t^i-\bar{v}_t^i|^2]}.
		\end{aligned}
	\end{align*}
\end{proof}
Therefore we have
\begin{align}
	\begin{aligned} \label{v}
		 \frac{d}{dt}E\left[ |v_t^i-\bar{v}_t^i|^2\right] \leq& \left(4\kappa \psi_M e^{-\frac{2C_mt}{3}}\sqrt{||f_0||_{L^1} ||(|v|^2+|x|^2)f_0||_{L^1}}\right.\\
		&\left.+4\kappa\psi_M (1\!+\!||f_0||_{L^1}\!)\sqrt{\frac{N-1}{N}}\sqrt{ E[|\bar{v}_t^i|^2]}\!+\!\frac{2\kappa\psi_M ||f_0||_{L^1}}{N}\sqrt{E[|\bar{v}_t^i|^2]}\right.\\
		&\left.+2\kappa\psi_m\sqrt{E[|\bar{v}_t^i|^2]}\right)\sqrt{E[|v_t^i-\bar{v}_t^i|^2]}\\
		&+2(\sigma-\kappa\psi_m) E[(v_t^i-\bar{v}_t^i)^2]\!-2E[(x_t^i-\bar{x}_t^i)\cdot\!(v_t^i-\bar{v}_t^i)].
	\end{aligned}
\end{align}
	\noindent $\bullet$ (Estimate for $E[(x_t^i-\bar{x}_t^i)\cdot(v_t^i-\bar{v}_t^i)]$): Now we estimate the cross term
	\begin{align*}
		\begin{aligned} %\label{}
			&d \left[(x_t^i-\bar{x}_t^i)\cdot(v_t^i-\bar{v}_t^i)\right]\\
			&= (v_t^i-\bar{v}_t^i)\cdot d(x_t^i-\bar{x}_t^i)+ (x_t^i-\bar{x}_t^i)\cdot d(v_t^i-\bar{v}_t^i)
			+d(x_t^i-\bar{x}_t^i)\cdot d(v_t^i-\bar{v}_t^i)\\
			&=(v_t^i-\bar{v}_t^i)^2dt-\frac{\kappa}{N}\sum_{j=1}^N\left[\psi(|x_t^j-x_t^i|)(v_t^i-v_t^j)-a(\bar{x}_t^i,t)\bar{v}_t^i+b(\bar{x}_t^i,t)\right]\cdot (x_t^i-\bar{x}_t^i)dt\\
			&\quad- (x_t^i-\bar{x}_t^i)^2dt+\sqrt{2\sigma}(x_t^i-\bar{x}_t^i)\cdot(v_t^i-\bar{v}_t^i)dW_t.
		\end{aligned}
	\end{align*}
We take the expectation to obtain
	\begin{align}
		\begin{aligned} \label{ff}
			& \frac{d}{dt}E\left[(x_t^i-\bar{x}_t^i)\cdot(v_t^i-\bar{v}_t^i)\right]\\
			&= E\left[|v_t^i-\bar{v}_t^i|^2\right]-E\left[|x_t^i-\bar{x}_t^i|^2\right]\\
			& \quad -\frac{\kappa}{N}E\left[\sum_{j=1}^N(\psi(|x_t^j-x_t^i|)(v_t^i-v_t^j)-a(\bar{x}_t^i,t)\bar{v}_t^i+b(\bar{x}_t^i,t))\cdot(x_t^i-\bar{x}_t^i)\right]\\
			&:=  E[|v_t^i-\bar{v}_t^i|^2]-E[|x_t^i-\bar{x}_t^i|^2] + \mathcal{I}_{4},
		\end{aligned}
	\end{align}
where we used
	\begin{align*}
	\begin{aligned} \label{}
		&E\left[(v_t^i-\bar{v}_t^i)\cdot(x_t^i-\bar{x}_t^i)dW_t\right]=0.
	\end{aligned}
\end{align*}

We further decompose the term  $ \mathcal{I}_{4}$ as follows:
	\begin{align*}
		\begin{aligned} %\label{}
			&  \mathcal{I}_{4} = -\frac{\kappa}{N}\! E\left[\sum_{j=1}^N(x_t^i-\bar{x}_t^i)\cdot\psi(|x_t^j-x_t^i|)\left((v_t^i-v_t^j)-(\bar{v}_t^i-\bar{v}_t^j)\right)\right]+\frac{\kappa}{N}E\left[(x_t^i-\bar{x}_t^i)\cdot a(\bar{x}_t^i,t)\bar{v}_t^i\right]\\
			& \qquad -\frac{\kappa}{N}E\left[\sum_{j\neq i}^N(x_t^i-\bar{x}_t^i)\cdot\psi(|x_t^j-x_t^i|)\left((\bar{v}_t^i-\bar{v}_t^j)-a(\bar{x}_t^i,t)\bar{v}_t^i\right)\right]-\frac{\kappa}{N}E\left[\sum\limits_{j=1}^N(x_t^i-\bar{x}_t^i)\cdot b(\bar{x}_t^i,t)\right]\\
			& \quad:= \mathcal{I}_{41} + \mathcal{I}_{42}+ \mathcal{I}_{43}+ \mathcal{I}_{44} .  
		\end{aligned}
	\end{align*}
\begin{lemma}
	The terms $\mathcal{I}_{4i}$, $i=1,\cdots,4$, satisfy the following estimates:
	\begin{align*}
		\begin{aligned}% \label{}
			& \mathcal{I}_{41}(t) \leq \frac{1}{2}E[|x_t^i-\bar{x}_t^i|^2]+2(\kappa \psi_M)^2 E[|v_t^i-\bar{v}_t^i|^2],\\
			& \mathcal{I}_{42}(t) \leq \frac{\kappa\psi_M||f_0||_{L^1}}{N}\sqrt{E[|x_t^i-\bar{x}_t^i|^2]}\sqrt{E[|\bar{v}_t^i|^2]},\\
			& \mathcal{I}_{43}(t) \leq 2\kappa \sqrt{\frac{N-1}{N}}\psi_M(1+||f_0||_{L^1})\sqrt{E[|x_t^i-\bar{x}_t^i|^2]}\sqrt{E[|\bar{v}_t^i|^2]} ,\\
			& \mathcal{I}_{44}(t) \leq 2\kappa \psi_M e^{-\frac{2C_mt}{3}}\sqrt{||f_0||_{L^1}||(|v|^2+|x|^2)f_0||_{L^1}}\sqrt{E[|x_t^i-\bar{x}_t^i|^2]}.
		\end{aligned}
	\end{align*}
\end{lemma}
\begin{proof}
	Now we first estimate $\mathcal{I}_{41}$ as follows:
	\begin{align*}
		\begin{aligned} %\label{}
			\mathcal{I}_{41}& = -\frac{\kappa}{N} E[\sum_{j=1}^N(x_t^i-\bar{x}_t^i)\cdot\psi(|x_t^i-x_t^j|)\left((v_t^i-v_t^j)-(\bar{v}_t^i-\bar{v}_t^j)\right)]\\
			&\leq \frac{\kappa \psi_M}{N}\sum_{j=1}^N E\left[|(x_t^i-\bar{x}_t^i)\cdot(v_t^i-\bar{v}_t^i)|+|(x_t^i-\bar{x}_t^i)\cdot(v_t^j-\bar{v}_t^j)|\right]\\
			&\leq \kappa \psi_M \left[\sqrt{E[ |x_t^i-\bar{x}_t^i|^2]} \sqrt{E[ |v_t^i-\bar{v}_t^i|^2]}+\sqrt{E[ |x_t^i-\bar{x}_t^i|^2]} \sqrt{E[ |v_t^j-\bar{v}_t^j|^2]}\right]\\
			%&\leq \kappa \psi_M \sqrt{E[\sum_{i=1}^N |x_t^i-\bar{x}_t^i|^2]} \sqrt{E[\sum_{i=1}^N |v_t^i-\bar{v}_t^i|^2]}+\frac{\kappa \psi_M}{N}E[\sum_{i=1}^N|x_t^i-\bar{x}_t^i|\sum_{j=1}^N|v_t^j-\bar{v}_t^j|]\\
			%&\leq \kappa \psi_M \sqrt{E[\sum_{i=1}^N |x_t^i-\bar{x}_t^i|^2]} \sqrt{E[\sum_{i=1}^N |v_t^i-\bar{v}_t^i|^2]}+\frac{\kappa \psi_M}{N}E[\frac{1}{4\kappa\psi_M}(\sum_{i=1}^N (x_t^i-\bar{x}_t^i))^2+\kappa\psi_M(\sum_{i=1}^N (v_t^i-\bar{v}_t^i))^2]\\
			&\leq \kappa \psi_M \left[\frac{1}{4\kappa \psi_M }E\left[|x_t^i-\bar{x}_t^i|^2\right]+\kappa \psi_M E\left[|v_t^i-\bar{v}_t^i|^2\right] +\frac{1}{4\kappa \psi_M}E\left[|x_t^i-\bar{x}_t^i|^2\right]+\kappa \psi_M E\left[|v_t^i-\bar{v}_t^i|^2\right]\right]\\
			&\leq \frac{1}{2}E\left[|x_t^i-\bar{x}_t^i|^2\right]+2(\kappa \psi_M)^2 E\left[|v_t^i-\bar{v}_t^i|^2\right],
		\end{aligned}
	\end{align*}
where we used the young equality.

Then we estimate $\mathcal{I}_{42}$:
	\begin{align*}
		\begin{aligned}% \label{}
			&  \mathcal{I}_{42} = \frac{\kappa}{N} E\left[(x_t^i-\bar{x}_t^i)\cdot a(\bar{x}_t^i,t)\bar{v}_t^i\right]
			%& \leq \frac{\kappa\psi_M||f_0||_{L^1}}{N}\sum_{i=1}^N \sqrt{E[|x_t^i-\bar{x}_t^i|^2]}\sqrt{E[ |\bar{v}_t^i|^2]}\\
\leq \frac{\kappa\psi_M||f_0||_{L^1}}{N}\sqrt{E[|\bar{v}_t^i|^2]}\sqrt{E[|x_t^i-\bar{x}_t^i|^2]}.
		\end{aligned}
	\end{align*}
\quad Next we use the same way to estimate $\mathcal{I}_{43}$. By \eqref{uu}, we obtain
%\begin{align}
%	\begin{aligned} \label{ww}
%		& \sum_{j=2}^N\!|\Psi^j|^2 \!\leq 2 \max \{| \psi(|x_t^1-x_t^j|)\!-\!a(\bar{x}_t^1,t)|,\psi(|x_t^1\!-\!x_t^j|)\}^2\!\left(\!\sum_{i=1}^N|\bar{v}_t^i|^2 \!+\!(N\!-\!2)|\bar{v}_t^1|^2\!\right).
%	\end{aligned}
%\end{align}
\begin{align*}
	\begin{aligned} \label{}
		 \mathcal{I}_{43} &= -\frac{\kappa}{N} E\left[\sum_{j= 2}^N(x_t^i-\bar{x}_t^i)\cdot \Psi^j\right] \leq \frac{\kappa}{N} \sqrt{E\left[|x_t^i-\bar{x}_t^i|^2\right]}\sqrt{E\left[\left|\sum_{j= 2}^N \Psi^j\right|^2\right]}\\
		&\leq \frac{\kappa \sqrt{N-1}}{N} \sqrt{E\left[|x_t^i-\bar{x}_t^i|^2\right]}\sqrt{\sum_{j= 2}^N E\left[|\Psi^j|^2\right]}\\
		%& \quad \leq \frac{2\kappa \psi_M (1+||f_0||_{L^1})\sqrt{N} \sqrt{N-1}}{N} \sum_{i= 1}^N \sqrt{E[|x_t^i-\bar{x}_t^i|^2]}\sqrt{ E[|\bar{v}_t^i|^2]}\\
		&\leq 2\kappa \sqrt{\frac{N-1}{N}}  \psi_M (1+||f_0||_{L^1}) \sqrt{E\left[|x_t^i-\bar{x}_t^i|^2\right]}\sqrt{ E\left[|\bar{v}_t^i|^2\right]}.\\
	\end{aligned}
\end{align*}

Finally we estimate $\mathcal{I}_{44}$:
	\begin{align*}
		\begin{aligned} %\label{}
			& |\mathcal{I}_{44}|=\left|\frac{\kappa}{N} E\left[\sum_{j=1}^N(x_t^i-\bar{x}_t^i)\cdot b(\bar{x}_t^i,t)\right]\right| \leq \kappa \sqrt{E\left[| b(\bar{x}_t^i,t)|^2\right]}\sqrt{E\left[|x_t^i-\bar{x}_t^i|^2\right]}\\
			%&\leq 2\kappa \psi_M e^{-\frac{C_mt}{2}}\sqrt{||f_0||_{L^1}||(|v|^2+|x|^2)f_0||_{L^1}}\sqrt{E[|(x_t^i-\bar{x}_t^i)|^2]}\\
			&\quad  \leq2 \kappa \psi_M e^{-\frac{2C_mt}{3}}\sqrt{||f_0||_{L^1}||(|v|^2+|x|^2)f_0||_{L^1}}\sqrt{E[|x_t^i-\bar{x}_t^i|^2]}.
		\end{aligned}
	\end{align*}
\end{proof}	
Therefore we have
	\begin{align}
	\begin{aligned} \label{xv}
	&\frac{d}{dt}E[(x_t^i-\bar{x}_t^i)\cdot (v_t^i-\bar{v}_t^i)]\leq (1+2(\kappa \psi_M)^2)E[|v_t^i-\bar{v}_t^i|^2]- \frac{1}{2}E[|x_t^i-\bar{x}_t^i|^2]\\
		&\quad+ \left(\frac{\kappa\psi_M||f_0||_{L^1}}{N}\sqrt{E[|\bar{v}_t^i|^2]} + 2\kappa \sqrt{\frac{N-1}{N}}  \psi_M (1+||f_0||_{L^1}) \sqrt{ E[|\bar{v}_t^i|^2]}\right.\\
		&\left.\quad+2\kappa \psi_M e^{-\frac{2C_mt}{3}}\sqrt{||f_0||_{L^1}||(|v|^2+|x|^2)f_0||_{L^1}}\right)\sqrt{E[|x_t^i-\bar{x}_t^i|^2]}.\\
	\end{aligned}
\end{align}

Now we take a combination $\frac{1}{2} \eqref{x} +\frac{1}{2} \eqref{v}+ \epsilon \eqref{xv} $ to obtain
	\begin{align*}
		\begin{aligned}% \label{}
		&\frac{d}{dt}E[\frac{1}{2}|x_t^i-\bar{x}_t^i|^2+\frac{1}{2}|v_t^i-\bar{v}_t^i|^2+\epsilon (x_t^i-\bar{x}_t^i)\cdot (v_t^i-\bar{v}_t^i)]\\
		&\leq -(\kappa\psi_m-\sigma-(1+2(\kappa \psi_M)^2)\epsilon)E[|v_t^i-\bar{v}_t^i|^2]-\frac{\epsilon}{2}E[|x_t^i-\bar{x}_t^i|^2]\\
		&\quad +\left[\left(2\kappa \psi_M (1+||f_0||_{L^1})\sqrt{\frac{N-1}{N}} +\frac{\kappa \psi_M ||f_0||_{L^1}}{N}+\kappa\psi_m\right)\sqrt{E[|\bar{v}_t^i|^2]}\right.\\
		&\left. \qquad+2\kappa \psi_M \sqrt{||f_0||_{L^1} ||(|v|^2+|x|^2)f_0||_{L^1}}e^{-\frac{2C_mt}{3}} \right]\sqrt{E[|v_t^i-\bar{v}_t^i|^2]}\\
		&\quad +\left[\left(2\kappa \psi_M (1+||f_0||_{L^1})\sqrt{\frac{N-1}{N}}+\frac{\kappa\psi_M||f_0||_{L^1}}{N} \right) \sqrt{ E[|\bar{v}_t^i|^2]}\right.\\
		&\left. \qquad +2\kappa \psi_M \sqrt{||f_0||_{L^1}||(|v|^2+|x|^2)f_0||_{L^1}}e^{-\frac{2C_mt}{3}} \right] \epsilon\sqrt{E[|x_t^i-\bar{x}_t^i|^2]}.
		\end{aligned}
	\end{align*}	
We take $\epsilon = \min\{\frac{\kappa\psi_m -\sigma}{2(1+2(\kappa\psi_M)^2)},\frac{1}{2} \} $ and by Lemma \ref{i5} we get
	\begin{align*}
		\begin{aligned} %\label{ii}
			&\frac{d}{dt}E\left[\frac{1}{2}|x_t^i-\bar{x}_t^i|^2+\frac{1}{2}|v_t^i-\bar{v}_t^i|^2+\epsilon (x_t^i-\bar{x}_t^i)\cdot (v_t^i-\bar{v}_t^i)\right]\\
			&\leq -C_1\left(E[|x_t^i-\bar{x}_t^i|^2]+E[|v_t^i-\bar{v}_t^i|^2]\right)\\
			&\quad+C_2(e^{-\frac{2}{3}C_m t}+e^{-\frac{2}{3}C_\ast t})\sqrt{\left(E[|x_t^i-\bar{x}_t^i|^2]+E[|v_t^i-\bar{v}_t^i|^2]\right)}\\
			%&\quad+C_2\left(\frac{2C_2}{C_1}(e^{-\gamma t}+e^{-\frac{4}{3}\eta t})+\frac{C_1}{2C_2}\left(E[|x_t^i-\bar{x}_t^i|^2]+E[|v_t^i-\bar{v}_t^i|^2]\right)\right),
		\end{aligned}
	\end{align*}
where $ C_1 = \min\{ \frac{\epsilon}{2}, \frac{\kappa\psi_m-\sigma}{2}\}$ and $C_2$ depends on $\kappa$, $\psi_M$ and $f_0$.
%where $C_3 = \max\{2\kappa \psi_M \sqrt{||f_0||_{L^1}||(|v|^2+|x|^2)f_0||_{L^1}},2\kappa \psi_M (1+||f_0||_{L^1})\sqrt{\frac{N-1}{N}} +\frac{\kappa \psi_M ||f_0||_{L^1}}{N}+\kappa\psi_m \}$and $ C_1 = \min\{ \frac{\epsilon}{2}, \frac{\kappa\psi_m-\sigma}{2}\}$. %$C_2=O(1)\kappa\psi_M(1+||f_0||_{L^1}+||(|v|^2+|x|^2)f_0||_{L^1})$.

Since $|\epsilon|<\frac{1}{2}$, we have 
\begin{align}
	\begin{aligned} \label{i}
    	&\frac{d}{dt}E\left[\frac{1}{2}|x_t^i-\bar{x}_t^i|^2+\frac{1}{2}|v_t^i-\bar{v}_t^i|^2+\epsilon (x_t^i-\bar{x}_t^i)\cdot (v_t^i-\bar{v}_t^i) \right]\\
		&\leq -\frac{4C_1}{3}E \left [\frac{1}{2}|x_t^i-\bar{x}_t^i|^2+\frac{1}{2}|v_t^i-\bar{v}_t^i|^2+\epsilon (x_t^i-\bar{x}_t^i)\cdot (v_t^i-\bar{v}_t^i)\right]\\
		&\quad+\frac{4\sqrt{3}}{3}C_2(e^{-\frac{2}{3}C_m  t}\!+\! e^{-\frac{2}{3}C_\ast t})\sqrt{E \left [\frac{1}{2}|x_t^i-\bar{x}_t^i|^2\!+\!\frac{1}{2}|v_t^i-\bar{v}_t^i|^2\!+\!\epsilon (x_t^i-\bar{x}_t^i)\!\cdot\! (v_t^i-\bar{v}_t^i)\right]}.
	\end{aligned}
\end{align}
We set 
\begin{align*}
	\begin{aligned} 
		& Z(t):= E\left[\frac{1}{2}|x_t^i-\bar{x}_t^i|^2+\frac{1}{2}|v_t^i-\bar{v}_t^i|^2+\epsilon (x_t^i-\bar{x}_t^i)\cdot (v_t^i-\bar{v}_t^i)\right].
	\end{aligned}
\end{align*}
Then, it follows from \eqref{i} that $Z(t)$ satisfies
\begin{align*}
	\begin{aligned} 
		& \frac{dZ(t)}{dt}\leq -\frac{4C_1}{3}Z(t)+\frac{4\sqrt{3}}{3}C_2(e^{-\frac{2}{3}C_m t}+e^{-\frac{2}{3}C_\ast t})\sqrt{Z(t)}.
	\end{aligned}
\end{align*}
%Now we set $\bar{Z}(t):=\sqrt{Z(t)}$
%and then we have 
%\begin{align*}
%	\begin{aligned} 
%		& \frac{d\bar{Z}(t)}{dt}\leq -\frac{2C_1}{3} \bar{Z}(t)+ \frac{2\sqrt{3}}{3}C_2(e^{-\frac{2}{3}C_m t}+e^{-\frac{2}{3}C_\ast t})
%	\end{aligned}
%\end{align*}
%which means
%\begin{align*}
%	\begin{aligned} %\label{}
%		&  \bar{Z}(t) \leq \frac{\sqrt{3}C_2}{C_1}(e^{-\frac{1}{3}C_m t}+e^{-\frac{1}{3}C_\ast t}) + \frac{2\sqrt{3}C_2}{C_1} e^{-\frac{1}{3}C_1 t}.
%	\end{aligned}
%\end{align*}
%Then we take $C_3= \min\{C_m, C_\ast, C_1\}$.
%$C_5=\max\{ \frac{2C_3}{C_1},4C_1C_3\}$.
By the general Gr\"onwall's inequality, we have
\begin{align*}
	\begin{aligned} %\label{}
		&  Z(t) \leq Ce^{-C_3 t},
	\end{aligned}
\end{align*}
%\begin{align*}
%	\begin{aligned} %\label{}
%		&  \bar{Z}(t) \leq \frac{2\sqrt{3}C_2}{C_1} e^{-C_3 t}
%	\end{aligned}
%\end{align*}
which means 
%where  $C_3$ depends on $C_1,C_2, C_m$.
%Therefore we have
	\begin{align*}
		\begin{aligned} %\label{}
			& E[|x_t^i-\bar{x}_t^i|^2]+E[|v_t^i-\bar{v}_t^i|^2] \leq Ce^{-C_3 t}
		\end{aligned}
	\end{align*}	
%\begin{align*}
%	\begin{aligned} %\label{}
%		& E[|x_t^i-\bar{x}_t^i|^2]+E[|v_t^i-\bar{v}_t^i|^2] \leq \frac{12(C_2)^2}{(C_1)^2}e^{-2C_3 t}.
%	\end{aligned}
%\end{align*}
	for some general constants $C$ and $C_3$ depend on initial configuration, $\psi$, $\kappa$ and $\sigma$.
	% Moreover, by symmetry,  $E(|x_t^i-\bar{x}_t^i|^2)$, $E(|v_t^i-\bar{v}_t^i|^2)$ are independent of $i=1,2,\cdots, N$, so by averaging on $i$,
	% we have 
	% \begin{align*}
	%	\begin{aligned}
		%	& E[|x_t^i-\bar{x}_t^i|^2]+E[|v_t^i-\bar{v}_t^i|^2] \leq \frac{1}{N}\frac{C_2^2}{C_1^2} (e^{-\gamma t}+e^{-\frac{\alpha t}{4}}).
		%\end{aligned}
	%\end{align*}
	%5 Hence we provide a decay rat in $N$ in the process of the mean-field limit.

\subsection{Finite-in-time propagation of chaos}
 In this subsection, we state the  local-in-time mean-field limit.
 \begin{theorem}\lbl{t4.2}
Suppose that the communication weight function $\psi$, $\kappa$, $\sigma$ and initial configuration $f_0$ satisfy the conditions: there exist positive constants $\psi_m, \psi_M, C$ such that
\begin{align*}
	\begin{aligned} %\label{}
		& 0< \psi_m \leq \psi \leq \psi_M, \quad \kappa\psi_m\min\{||f_0||_{L^1},1\}>d\sigma, \quad  \int_{\mathbb{R}^{2d}}(e^{C|v|^2}+x^2)f_0dxdv<\infty
	\end{aligned}
\end{align*}
and let $(x_t^i,v_t^i)$ and $(\bar{x}_t^i,\bar{v}_t^i,f)$ be the solution processes to the systems \eqref{SCS} and \eqref{MK2} respectively.
Then,  for any finite time interval $[0,T]$ and $N\ge 1$, we have 
	\begin{align*}
		\begin{aligned}
			&E[|x_t^i-\bar{x}_t^i|^2]+E[|v_t^i-\bar{v}_t^i|^2] \leq \frac{C}{N^{e^{-Ct}}},
		\end{aligned}
	\end{align*}
where $C$ is a general positive constant independent  of $N$.
\end{theorem}
  The local-in-time mean-field limit can  be constructed by using a similar argument in Theorem 1.1 in \cite{BC}. Here we include some details for the sake of the reader. 
 \begin{proof}
 \noindent $\bullet$ (Estimate for $\frac{d}{dt}E\left[|x_t^i-\bar{x}_t^i|^2\right]$): By \eqref{x}, we have
	\begin{align}
		\begin{aligned} \label{l-x}
			&\frac{d}{dt}E[|x_t^i-\bar{x}_t^i|^2] = 2 E[(x_t^i-\bar{x}_t^i)\cdot (v_t^i-\bar{v}_t^i)].
		\end{aligned}
	\end{align}
	\noindent $\bullet$ (Re-estimate for $\frac{d}{dt}E\left[|v_t^i-\bar{v}_t^i|^2\right]$): By \eqref{hh}, we have
              \begin{align*}
		\begin{aligned} 
			&\frac{d}{dt} E\left[|v_t^i-\bar{v}_t^i|^2\right]\\
			& =-\frac{2\kappa}{N}E\left[\sum_{j=1}^N(v_t^i-\bar{v}_t^i)\cdot\left(\psi(|x_t^j-x_t^i|) (v_t^i-v_t^j)-a(\bar{x}_t^i,t)\bar{v}_t^i+b(\bar{x}_t^i,t)\right)\right]\\
			&\quad +2\sigma E\left[|v_t^i-\bar{v}_t^i|^2\right]-2E\left[(x_t^i-\bar{x}_t^i)(v_t^i-\bar{v}_t^i)\right]\\
			&:= \mathcal{I}_3+2\sigma E\left[|v_t^i-\bar{v}_t^i|^2\right]-2E\left[(x_t^i-\bar{x}_t^i)\cdot(v_t^i-\bar{v}_t^i)\right].
		\end{aligned}
	 \end{align*}
Here we re-decompose the term $\mathcal{I}_3$ as follows:
	\begin{align*}
		\begin{aligned} 
			\mathcal{I}_3 =& -\frac{2\kappa}{N}E\left[\sum_{j=1}^N(v_t^i-\bar{v}_t^i)\cdot\left(\psi(|x_t^j-x_t^i|) (v_t^i-v_t^j)-\psi(|\bar{x}_t^j-\bar{x}_t^i|)(\bar{v}_t^i-\bar{v}_t^j)\right)\right]\\
			&+\frac{2\kappa}{N} E\left[(v_t^i-\bar{v}_t^i) \cdot\left( 0-\psi\ast f(\bar{x}_t^j,\bar{v}_t^i)\right)\right]\\
			&-\frac{2\kappa}{N}E\left[\sum_{j\neq i}^N(v_t^i-\bar{v}_t^i)\cdot \left(\psi(|\bar{x}_t^j-\bar{x}_t^i|)(\bar{v}_t^i-\bar{v}_t^j)-\psi\ast f(\bar{x}_t^j,\bar{v}_t^i)\right)\right]\\
			 \quad:=& \mathcal{I}_{3a}+\mathcal{I}_{3b}+\mathcal{I}_{3c}.
		\end{aligned}
	\end{align*}
	
	Similar to  \cite{BC}, we conclude that given $T > 0$, there exists $C>0$ such that
	\begin{align*}
	\begin{aligned}
	 &\mathcal{I}_{3a}\le C(1+r)	\mathfrak L(t)+Ce^{-r},\\
	 & \mathcal{I}_{3b}\le \frac{C}{N}\sqrt{\mathfrak L(t)}
	\end{aligned}
	\end{align*}
	for all $r>0$ and all $0\le t\le T$.
	
	 The term $\mathcal{I}_{3b}$ can be  treated  as follows by a law of large numbers argument.
	  By symmetry  that the quantity  is independent of the label $i$, we assume that $i=1$. We start by applying that Cauchy-Schwartz inequality to obtain  
	\begin{align*}
	\begin{aligned}
	  &\mathcal{I}_{3c}\le\frac{1}{N}\sqrt{E[|v_t^i-\bar{v}_t^i|^2]}\sqrt{E\left[\left(\sum_{j=2}^N Y^j\right)^2\right]},
	\end{aligned}
	\end{align*}
	where
	$Y^j=\psi(|\bar{x}_t^j-\bar{x}_t^i|)(\bar{v}_t^i-\bar{v}_t^j)-\psi\ast f(\bar{x}_t^j,\bar{v}_t^i)$ for $j\ge 2$. Note that for $j\ne k$, by independence of the $N$ process $(\bar{x}_t^i,\bar{v}_t^i)$ and the same probability distribution $f$, 
	we  have $E[Y^j\cdot Y^i]=0$.
	 Then 
	 \begin{align*}
	\begin{aligned}
	  & \mathcal{I}_{3c}\!\le\!\frac{1}{N}\sqrt{E[|v_t^i-\bar{v}_t^i|^2]}\sqrt{E\left[\left(\sum_{j=2}^N Y^j\right)^2\right]}\\
	  &\quad \le\frac{1}{N}\sqrt{E[|v_t^i-\bar{v}_t^i|^2]}\sqrt{(N\!-\!1)E[|Y^2|^2]}\\
	  &\quad \le \frac{C}{\sqrt{N}}\sqrt{\mathfrak L(t)}.
	\end{aligned}
	\end{align*}

	 Hence, we have 
	  \begin{align}
		\begin{aligned}  \label{l-v}
			&\frac{d}{dt} E[|v_t^i-\bar{v}_t^i|^2] =-\frac{2\kappa}{N}E\left[\sum_{j=1}^N(v_t^i-\bar{v}_t^i)\cdot \left(\psi(|x_t^j-x_t^i|) (v_t^i-v_t^j)-a(\bar{x}_t^i,t)\bar{v}_t^i+b(\bar{x}_t^i,t)\right)\right]\\
			&\quad +2\sigma E[|v_t^i-\bar{v}_t^i|^2]-2E[(x_t^i-\bar{x}_t^i)\cdot (v_t^i-\bar{v}_t^i)]\\
			&\leq 2\sigma E[|v_t^i\!-\!\bar{v}_t^i|^2]\!-\!2E[(x_t^i\!-\!\bar{x}_t^i)\cdot(v_t^i\!-\!\bar{v}_t^i)]\!+\!C(1\!+\!r)\mathfrak{L}(t)\!+\!Ce^{-r}\!+\!\frac{C\sqrt{\mathfrak L(t)}}{\sqrt{N}}.
		\end{aligned}
	 \end{align}
	\noindent $\bullet$ (Re-estimate for $\frac{d}{dt}E[(x_t^i-\bar{x}_t^i)\cdot (v_t^i-\bar{v}_t^i)]$): By \eqref{ff}, we have 
	     \begin{align*}
		\begin{aligned}% \label{ff}
			& \frac{d}{dt}E[(x_t^i-\bar{x}_t^i)\cdot (v_t^i-\bar{v}_t^i)]= E[|v_t^i-\bar{v}_t^i|^2]-E[|x_t^i-\bar{x}_t^i|^2]\\
			& \quad -\frac{\kappa}{N}E[\sum_{j=1}^N(\psi(|x_t^j-x_t^i|)(v_t^i-v_t^j)-a(\bar{x}_t^i,t)\bar{v}_t^i+b(\bar{x}_t^i,t))\cdot (x_t^i-\bar{x}_t^i)]\\
			&:=  E[|v_t^i-\bar{v}_t^i|^2]-E[|x_t^i-\bar{x}_t^i|^2] + \mathcal{I}_{4}
		\end{aligned}
	\end{align*}
 and now we re-decompose the term  $ \mathcal{I}_{4}$ as follows:
	\begin{align*}
		\begin{aligned} 
			\mathcal{I}_4 =& -\frac{\kappa}{N}E\left[\sum_{j=1}^N(x_t^i-\bar{x}_t^i)\cdot(\psi(|x_t^j-x_t^i|) (v_t^i-v_t^j)-\psi(|\bar{x}_t^j-\bar{x}_t^i|)(\bar{v}_t^i-\bar{v}_t^j))\right]\\
			&+\frac{\kappa}{N} E[(x_t^i-\bar{x}_t^i)\cdot ( 0-\psi\ast f(\bar{x}_t^j,\bar{v}_t^i))]\!-\!\frac{\kappa}{N}E\left[\sum_{j\neq i}^N(x_t^i\!-\!\bar{x}_t^i)\!\cdot\!(\psi(|\bar{x}_t^j-\bar{x}_t^i|)(\bar{v}_t^i-\bar{v}_t^j)\!-\!\psi\ast f(\bar{x}_t^j,\bar{v}_t^i))\right]\\
			 \quad:=& \mathcal{I}_{4a}+\mathcal{I}_{4b}+\mathcal{I}_{4c}.
		\end{aligned}
	\end{align*}
Then we get the following estimate
	 \begin{align}
		\begin{aligned} \label{l-xv}
			& \frac{d}{dt}E[(x_t^i-\bar{x}_t^i)\cdot (v_t^i-\bar{v}_t^i)]\\
			&\le  E[|v_t^i-\bar{v}_t^i|^2]-E[|x_t^i-\bar{x}_t^i|^2] +C(1+r)\mathfrak{L}(t)+Ce^{-r}+\frac{C}{\sqrt{N}}\sqrt{\mathfrak{L}(t)}.
		\end{aligned}
	\end{align}
	
	 Combining \eqref{l-x}, \eqref{l-v} and \eqref{l-xv}, we have 
 \begin{align*}
		\begin{aligned} 
			 \frac{d}{dt}\tilde{\mathfrak{L}}(t)&\le  C(1+r)\mathfrak{L}(t)+Ce^{-r}+\frac{C}{\sqrt{N}}\sqrt{\mathfrak{L}(t)}\\
			 &\le C(1+r)\tilde{\mathfrak L}(t)+Ce^{-r}+\frac{C}{N}.
		\end{aligned}
	\end{align*}
	 Therefore, by the proof of Theorem 1.1 in \cite{BC}  we prove that
	 \begin{align*}
		\begin{aligned} 
		   \tilde{\mathfrak L}(t)\le  CN^{-e^{-Ct}},
		\end{aligned}
	\end{align*}
	which is equivalent to   \begin{align*}
		\begin{aligned} 
			\mathfrak L(t) \le  CN^{-e^{-Ct}}
		\end{aligned}
	\end{align*}
	 for $0\le t\le T$. 
 \end{proof}
 \subsection{Uniform-in-time mean-field limit}
  We are now ready to state the main result for this section:
  \begin{theorem}
Suppose that the communication weight function $\psi,\kappa,\sigma$ and initial configuration $f_0$ satisfy the conditions: there exist positive constants $\psi_m, \psi_M, C$ such that
\begin{align*}
	\begin{aligned} %\label{}
		& 0 < \psi_m \leq \psi \leq \psi_M, \quad \kappa\psi_m\min\{||f_0||_{L^1},1\}> d\sigma, \quad  \int_{\mathbb{R}^{2d}}(e^{C|v|^2}+x^2)f_0dxdv<\infty
	\end{aligned}
\end{align*}
and let $(x_t^i,v_t^i)$ and $(\bar{x}_t^i,\bar{v}_t^i,f)$ be the solution processes to the systems \eqref{SCS} and \eqref{MK2} respectively.
Then  we have 
	\begin{align*}
		\begin{aligned} 
			 \varlimsup\limits_{N\to+\infty}\sup\limits_{0\le t<+\infty}(E[|x_t^i-\bar{x}_t^i|^2]+E[|v_t^i-\bar{v}_t^i|^2])=0.
		\end{aligned}
	\end{align*}
%where $C$ is a general positive constant.
\end{theorem}
\begin{proof}
 We can prove this theorem by a contradiction argument. Suppose  that
\begin{align}
		\begin{aligned} \label{21}
			 \varlimsup\limits_{N\to+\infty}\sup\limits_{0\le t<+\infty}(E[|x_t^i-\bar{x}_t^i|^2]+E[|v_t^i-\bar{v}_t^i|^2])=D_1>0.
		\end{aligned}
	\end{align}
	  By Theorem \ref{4.1}, we choose a constant  $T_0$ such that, for any $N$
	  \begin{align}
	  	\begin{aligned} \label{22}
	  		\sup\limits_{T_0\le t<+\infty}(E[|x_t^i-\bar{x}_t^i|^2]+E[|v_t^i-\bar{v}_t^i|^2])\le Ce^{-C_m T_0}\le\frac{D_1}{2}.
	  	\end{aligned}
	  \end{align}
	 Combining \eqref{21} and \eqref{22}, we know
	     \begin{equation*}
	 \varlimsup\limits_{N\to+\infty} \sup\limits_{0\le T\le T_0}(E[|x_t^i-\bar{x}_t^i|^2]+E[|v_t^i-\bar{v}_t^i|^2])= D_1,
	  \end{equation*}
	  which contradicts to Theorem \ref{t4.2}. Therefore, we conclude  $D_1=0$.
\end{proof}

\section{Acknowledgment}
This work  is supported by Natural Science Foundation of China No.12001097 and Fundamental Research Funds for the Central Universities.
The first author would like to  express the gratitude to Professor Seung-Yeal Ha from Seoul National  University  for  taking her into this research area.

 \bibliography{references}

\end{document}